\documentclass[11pt]{article}
\usepackage{amsmath,amsthm,amssymb}
\usepackage[usenames,dvipsnames]{xcolor}
\usepackage{enumerate}
\usepackage{graphicx}
\usepackage{cite}
\usepackage{comment}
\usepackage{oands}
\usepackage{tikz}
\usepackage{changepage}
\usepackage{bbm}
\usepackage{mathtools}
\usepackage[margin=1in]{geometry}
\usepackage[pagewise,mathlines]{lineno}
\usepackage{appendix}
\usepackage{stmaryrd}
\usepackage{multicol}
\usepackage{microtype}
\usepackage[colorinlistoftodos]{todonotes}

\usepackage[pdftitle={The critical Liouville quantum gravity metric induces the Euclidean topology},
  pdfauthor={Jian Ding and Ewain Gwynne},
colorlinks=true,linkcolor=NavyBlue,urlcolor=RoyalBlue,citecolor=PineGreen,bookmarks=true,bookmarksopen=true,bookmarksopenlevel=2,unicode=true,linktocpage]{hyperref}

\theoremstyle{plain}
\newtheorem{thm}{Theorem}[section]

\newtheorem{lem}[thm]{Lemma}
\newtheorem{prop}[thm]{Proposition}

\def\@rst #1 #2other{#1}
\newcommand\MR[1]{\relax\ifhmode\unskip\spacefactor3000 \space\fi
  \MRhref{\expandafter\@rst #1 other}{#1}}
\newcommand{\MRhref}[2]{\href{http://www.ams.org/mathscinet-getitem?mr=#1}{MR#2}}

\theoremstyle{definition}
\newtheorem{defn}[thm]{Definition}
\newtheorem{remark}[thm]{Remark}

\numberwithin{equation}{section}

\newcommand{\dsb}{\begin{adjustwidth}{2.5em}{0pt}
\begin{footnotesize}}
\newcommand{\dse}{\end{footnotesize}
\end{adjustwidth}}

\newcommand{\ssb}{\begin{adjustwidth}{2.5em}{0pt}}
\newcommand{\sse}{\end{adjustwidth}}

\newcommand{\aryb}{\begin{eqnarray*}}
\newcommand{\arye}{\end{eqnarray*}}
\def\alb#1\ale{\begin{align*}#1\end{align*}}
\def\allb#1\alle{\begin{align}#1\end{align}}
\newcommand{\eqb}{\begin{equation}}
\newcommand{\eqe}{\end{equation}}
\newcommand{\eqbn}{\begin{equation*}}
\newcommand{\eqen}{\end{equation*}}

\newcommand{\BB}{\mathbbm}
\newcommand{\ol}{\overline}

\newcommand{\op}{\operatorname}

\newcommand{\frk}{\mathfrak}
\newcommand{\eqD}{\overset{d}{=}}
\newcommand{\ep}{\varepsilon}
\newcommand{\rta}{\rightarrow}

\newcommand{\wt}{\widetilde}

\newcommand{\mcl}{\mathcal}

\newcommand{\bdy}{\partial}
\newcommand{\rng}{\mathring}

\newcommand{\crit}{{\mathrm{c}}}

\newcommand{\ccM}{{\mathbf{c}_{\mathrm M}}}

\let\originalleft\left
\let\originalright\right
\renewcommand{\left}{\mathopen{}\mathclose\bgroup\originalleft}
\renewcommand{\right}{\aftergroup\egroup\originalright}

\title{The critical Liouville quantum gravity metric induces the Euclidean topology}
 \date{ }
 \author{
\begin{tabular}{c} Jian Ding\\[-5pt]\small University of Pennsylvania \end{tabular}
\begin{tabular}{c} Ewain Gwynne\\[-5pt]\small University of Chicago \end{tabular} 
}

\begin{document}

\maketitle

\begin{abstract}
We show that every possible metric associated with critical ($\gamma=2$) Liouville quantum gravity (LQG) induces the same topology on the plane as the Euclidean metric. More precisely, we show that the optimal modulus of continuity of the critical LQG metric with respect to the Euclidean metric is a power of $1/\log(1/|\cdot|)$. Our result applies to every possible subsequential limit of critical Liouville first passage percolation, a natural approximation scheme for the LQG metric which was recently shown to be tight.  
\end{abstract}

\tableofcontents

\section{Introduction}
\label{sec-intro}

\subsection{Overview}
\label{sec-overview}

The \emph{Gaussian free field} (GFF) is the most natural random generalized function on a planar domain. 
There are many different versions of the GFF, corresponding, e.g., to different choices of domain and boundary conditions.
For concreteness, in this paper we will focus on the whole-plane GFF normalized so that its average over the unit circle is zero.
This is the centered Gaussian process $h$ on $\BB C$ with covariance function\footnote{See, e.g.,~\cite[Section 2.1.1]{vargas-dozz-notes} for a computation of this covariance function.}
\eqbn
\op{Cov}(h(z) , h(w)) = G(z,w) := \log \frac{\max\{|z|,1\}  \max\{|w|,1\}}{|z-w|} ,\quad\forall z,w\in \BB C  ,
\eqen 
interpreted as a random generalized function on $\BB C$. We refer to~\cite{shef-gff,pw-gff-notes} for more background on the GFF.

Liouville quantum gravity (LQG) is a class of models of random geometry defined using the exponential of the GFF, $e^h$. 
The exponential of the GFF does not make literal sense since $h$ is a generalized function, not a true function. 
However, one can rigorously define various objects associated with LQG by approximating $h$ by a family of continuous functions, then taking appropriate limits. 
In this paper, we will be primarily interested in the LQG metric (distance function), whose definition we review just below. 
We refer to~\cite{berestycki-lqg-notes,gwynne-ams-survey} for introductory expository articles on LQG focusing on aspects relevant to the present paper.

\subsubsection{Liouville first passage percolation}

To construct the LQG metric, let us first introduce a family of continuous functions which approximate $h$. 
For $s > 0$ and $z,w\in\BB C$, let $p_s(z) = \frac{1}{2\pi s} \exp\left( - \frac{|z|^2}{2s} \right)$ be the heat kernel. For $\ep >0$, we define a mollified version of the GFF by
\eqb \label{eqn-gff-convolve}
h_\ep^*(z) := (h*p_{\ep^2/2})(z) = \int_{\BB C} h(w) p_{\ep^2/2} (z-w) \, dw ,\quad \forall z\in \BB C  ,
\eqe
where the integral is interpreted in the sense of distributional pairing. We use $p_{\ep^2/2}$ instead of $p_\ep$ so that the variance of $h_\ep^*(z)$ is $\log\ep^{-1} + O_\ep(1)$. 
 
We now consider a parameter $\xi > 0$. \emph{Liouville first passage percolation} (LFPP) with parameter $\xi$ is the family of random metrics $\{D_h^\ep\}_{\ep > 0}$ defined by
\eqb \label{eqn-lfpp}
D_h^\ep(z,w) := \inf_{P : z\rta w} \int_0^1 e^{\xi h_\ep^*(P(t))} |P'(t)| \,dt ,\quad \forall z,w\in\BB C 
\eqe
where the infimum is over all piecewise continuously differentiable paths $P:[0,1]\rta\BB C$ from $z$ to $w$.  
To extract a non-trivial limit of the metrics $D_h^\ep$, we need to re-normalize. We (somewhat arbitrarily) define our normalizing factor by
\eqb \label{eqn-gff-constant}
\frk a_\ep := \text{median of} \: \inf\left\{ \int_0^1 e^{\xi h_\ep^*(P(t))} |P'(t)| \,dt  : \text{$P$ is a left-right crossing of $[0,1]^2$} \right\} ,
\eqe  
where a left-right crossing of $[0,1]^2$ is a piecewise continuously differentiable path $P : [0,1]\rta [0,1]^2$ joining the left and right boundaries of $[0,1]^2$.
 
It was shown in~\cite[Proposition 1.1]{dg-supercritical-lfpp} that for each $\xi > 0$, there exists $Q = Q(\xi) > 0$ such that 
\eqb  \label{eqn-Q-def}
\frk a_\ep = \ep^{1 - \xi Q  + o_\ep(1)} ,\quad \text{as} \quad \ep \rta 0 . 
\eqe 
The existence of $Q$ is proven via a subadditivity argument, so the exact relationship between $Q$ and $\xi$ is not known. However, it is known that $Q \in (0,\infty)$ for all $\xi  > 0$ and $Q$ is a non-increasing function of $\xi$~\cite{dg-supercritical-lfpp,lfpp-pos}. See also~\cite{gp-lfpp-bounds,ang-discrete-lfpp} for bounds for $Q$ in terms of $\xi$. 

We define the \emph{critical value} for the parameter $\xi$ by
\eqb
\xi_\crit := \inf\{\xi > 0 : Q(\xi) = 2\} .
\eqe
It follows from~\cite[Proposition 1.1]{dg-supercritical-lfpp} that $\xi_\crit$ is the unique value of $\xi$ for which $Q(\xi) = 2$ and from~\cite[Theorem 2.3]{gp-lfpp-bounds} that $\xi_\crit  \in [0.4135 , 0.4189]$. We have $Q > 2$ for $\xi < \xi_\crit$ and $Q \in(0,2)$ for $\xi > \xi_\crit$. 
 
\begin{defn}
We refer to LFPP with $\xi <\xi_\crit$, $\xi = \xi_\crit$, and $\xi > \xi_\crit$ as the \emph{subcritical}, \emph{critical}, and \emph{supercritical} phases, respectively.
\end{defn}

\begin{figure}[ht!]
\begin{center}
\includegraphics[width=\textwidth]{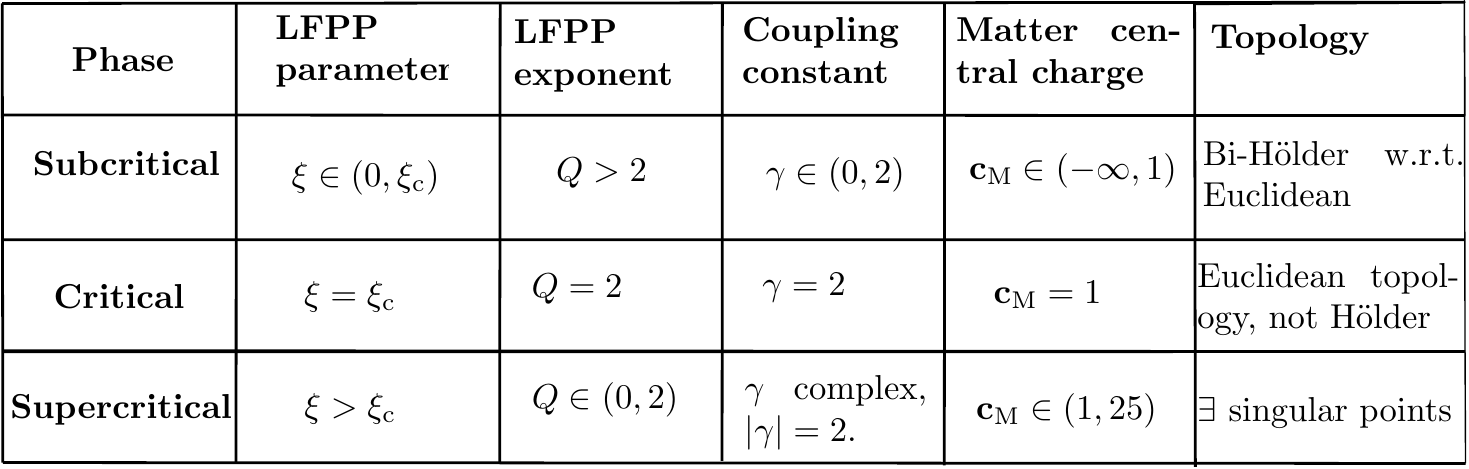}
\caption{\label{fig-c-phases-table} Table summarizing the phases for the LQG metric. The fact that the metric induces the Euclidean topology in the critical case is the main result of this paper. 
}
\end{center}
\end{figure}

\subsubsection{Subcritical and supercritical phases}

We will primarily be interested in the critical phase, but by way of context we will now discuss what happens in the subcritical and supercritical phases. See Figure~\ref{fig-c-phases-table} for a table which summarizes the three phases. 

In the subcritical phase, it was shown by Ding, Dub\'edat, Dunlap, and Falconet~\cite{dddf-lfpp} that the re-scaled LFPP metrics $\frk a_\ep^{-1} D_h^\ep$ are tight with respect to the topology of uniform convergence on compact subsets of $\BB C\times \BB C$. Every possible subsequential limit $D_h$ is a metric which induces the same topology on $\BB C$ as the Euclidean metric. 
Subsequently, it was shown by Gwynne and Miller~\cite{gm-uniqueness} (building on~\cite{local-metrics,lqg-metric-estimates,gm-confluence}) that the subsequential limit is unique. 

The limiting metric $D_h$ is the metric associated with LQG with coupling constant $\gamma \in (0,2)$, where $\gamma$ is related to $\xi$ by the non-explicit formulas
\eqb \label{eqn-gamma-xi}
Q(\xi) = \frac{2}{\gamma} + \frac{\gamma}{2} \quad \text{or equivalently} \quad \gamma = \xi d(\xi) .
\eqe
Here, $d(\xi)$ is the Hausdorff dimension of the metric space $(\BB C , D_h)$. Equivalently, $D_h$ is the metric associated to LQG with \emph{matter central charge} 
\eqb \label{eqn-c-xi}
\ccM = 25 - 6Q(\xi)^2  ,
\eqe
which lies in $(-\infty,1)$. 
 

In the supercritical and critical phases, we showed in~\cite{dg-supercritical-lfpp} that the metrics $\frk a_\ep^{-1} D_h^\ep$ are tight with respect to the topology on lower semicontinuous functions on $\BB C\times \BB C$, which we recall in Definition~\ref{def-lsc} below. Every possible subsequential limit is a metric on $\BB C$, except that it is allowed to take on infinite values. We expect that the subsequential limit is unique, but this has not yet been proven.
 
In the supercritical case, if $D_h$ is a subsequential limiting metric, then there is an uncountable, dense, Lebesgue measure zero set of \emph{singular points} $z\in \BB C$ for which
\eqb \label{eqn-singular-pt}
D_h(z,w) = \infty,\quad \forall w\in\BB C\setminus\{z\} .
\eqe
In particular, $D_h$ does not induce the same topology as the Euclidean metric. Nevertheless, a.s.\ the $D_h$-distance between any two non-singular points is finite, so $D_h(z,w)$ is finite for Lebesgue-a.e.\ pair of points $z,w\in\BB C$. 
Roughly speaking, the singular points correspond to points $z\in\BB C$ of ``thickness" greater than $Q$, i.e., points for which
\eqb \label{eqn-thick}
\limsup_{\ep\rta 0} \frac{h_\ep(z)}{\log\ep^{-1}  }  > Q,
\eqe
where $h_\ep(z)$ is the average of $h$ over the circle of radius $\ep$ centered at $z$~\cite[Proposition 1.11]{pfeffer-supercritical-lqg}. 

By extending the first formula in~\eqref{eqn-gamma-xi} and the formula~\eqref{eqn-c-xi}, we see that the supercritical case corresponds to LQG with $\gamma\in\BB C$ satisfying $|\gamma|=2$ or equivalently with $\ccM \in (1,25)$. LQG in this phase is much less well-understood than in the phase when $\ccM \leq 1$, even from a physics perspective. We refer to~\cite{ghpr-central-charge} for further discussion of LQG with $\ccM \in (1,25)$. 

\subsubsection{The critical case}

LFPP with $\xi = \xi_\crit$ corresponds to LQG with $\gamma = 2$ or equivalently $\ccM = 1$. 
This case is covered by the tightness result of~\cite{dg-supercritical-lfpp}, but estimates in the existing literature are not precise enough to determine whether there exist singular points for $\xi =\xi_\crit$. 
One reason for this is as follows. As noted above, singular points correspond to points of thickness greater than $Q$ for the GFF. For $\xi=\xi_\crit$ we have $Q = 2$. The value $\alpha = 2$ is critical for the existence of $\alpha$-thick points of $h$~\cite{hmp-thick-pts}, meaning that for for each $\alpha \in [-2,2]$ there exist points $z\in \BB C$ such that
\eqb \label{eqn-thick'}
\limsup_{\ep\rta 0} \frac{h_\ep(z)}{\log\ep^{-1}  } = \alpha,
\eqe 
but such points do not exist when $|\alpha| > 2$. Hence $\xi_\crit$ is exactly the critical threshold for singular points to exist.  

The purpose of this paper is to show that for $\xi = \xi_\crit$ there are no singular points, and that the limiting metric for $\xi = \xi_\crit$ induces the same topology as the Euclidean metric (Theorem~\ref{thm-holder}). 

Our result fits into a substantial exiting literature on critical Liouville quantum gravity. The LQG area measure has been constructed in the critical case $\gamma = 2$~\cite{shef-renormalization,shef-deriv-mart}, but the construction is more difficult than for $\gamma \in (0,2)$. We refer~\cite{powell-gmc-survey} for a survey of results on the critical LQG area measure. Critical LQG is also connected to Schramm-Loewner evolution (SLE$_\kappa$) at the critical value $\kappa = 4$~\cite{hp-welding}. 

One of the motivations for considering critical LQG is that (like subcritical LQG) it is expected to describe the scaling limit of various random planar maps. One of the conjectured modes of convergence is that certain random planar maps, equipped with the re-scaled graph distance, should converge to LQG surfaces equipped with the critical LQG metric with respect to the Gromov-Hausdorff topology. So far, this type of convergence has been proven only for uniform random planar maps toward LQG with $\gamma=\sqrt{8/3}$~\cite{legall-uniqueness,miermont-brownian-map,lqg-tbm1,lqg-tbm2}. 

Random planar map models which are conjectured to converge to critical ($\gamma=2$) LQG in the above sense include planar maps sampled with probability proportional to the partition function of the discrete Gaussian free field, the double dimer model, the four state Potts model, or the Fortuin-Kasteleyn model with parameter $q=4$~\cite{shef-burger}. The result of this paper suggests that the aforementioned random planar map models should have Gromov-Hausdorff scaling limits which are topological surfaces. We refer to~\cite{ghs-mating-survey} for a survey of the connections between random planar maps and LQG. 
\bigskip

\noindent \textbf{Acknowledgments.} We thank Jason Miller for helpful discussions. J.D.\ was partially supported by NSF grants DMS-1757479 and DMS-1953848. E.G.\ was partially supported by a Clay research fellowship. 

\subsection{Definition of a weak LQG metric}
\label{sec-metric-def}

The results of this paper hold not only for subsequential limits of LFPP, but also for a wider class of metrics called \emph{weak LQG metrics}.
Such metrics are defined in terms of a list of axioms which was first stated in~\cite{pfeffer-supercritical-lqg} (a similar list of axioms in the subcritical case was introduced earlier in~\cite{lqg-metric-estimates}). 
In this subsection, we will review the definition of a weak LQG metric. 
We first need a few preliminary definitions.

\begin{defn} \label{def-lsc} 
Let $X\subset \BB C$. 
A function $f : X \times X \rta \BB R \cup\{-\infty,+\infty\}$ is \emph{lower semicontinuous} if whenever $(z_n,w_n) \in X\times X$ with $(z_n,w_n) \rta (z,w)$, we have $f(z,w) \leq \liminf_{n\rta\infty} f(z_n,w_n)$. 
The \emph{topology on lower semicontinuous functions} is the topology whereby a sequence of such functions $\{f_n\}_{n\in\BB N}$ converges to another such function $f$ if and only if
\begin{enumerate}[(i)]
\item Whenever $(z_n,w_n) \in X\times X$ with $(z_n,w_n) \rta (z,w)$, we have $f(z,w) \leq \liminf_{n\rta\infty} f_n(z_n,w_n)$.
\item For each $(z,w)\in X\times X$, there exists a sequence $(z_n,w_n) \rta (z,w)$ such that $f_n(z_n,w_n) \rta f(z,w)$. 
\end{enumerate}
\end{defn}

It follows from~\cite[Lemma 1.5]{beer-usc} that the topology of Definition~\ref{def-lsc} is metrizable (see~\cite[Section 1.2]{dg-supercritical-lfpp}). 
Furthermore,~\cite[Theorem 1(a)]{beer-usc} shows that this metric can be taken to be separable. 

\begin{defn} \label{def-metric-properties}
Let $(X,d)$ be a metric space, with $d$ allowed to take on infinite values. 
\begin{itemize}
\item
For a curve $P : [a,b] \rta X$, the \emph{$d$-length} of $P$ is defined by 
\eqbn
\op{len}(P;d) :=  \sup_{T} \sum_{i=1}^{\# T} d(P(t_i) , P(t_{i-1})) 
\eqen
where the supremum is over all partitions $T : a= t_0 < \dots < t_{\# T} = b$ of $[a,b]$. Note that the $d$-length of a curve may be infinite.
\item
We say that $(X,d)$ is a \emph{length space} if for each $x,y\in X$ and each $\ep > 0$, there exists a curve of $d$-length at most $d(x,y) + \ep$ from $x$ to $y$. 
A curve from $x$ to $y$ of $d$-length \emph{exactly} $d(x,y)$ is called a \emph{geodesic}. 
\item
For $Y\subset X$, the \emph{internal metric of $d$ on $Y$} is defined by
\eqb \label{eqn-internal-def}
d(x,y ; Y)  := \inf_{P \subset Y} \op{len}\left(P ; d \right) ,\quad \forall x,y\in Y 
\eqe 
where the infimum is over all paths $P$ in $Y$ from $x$ to $y$. 
Note that $d(\cdot,\cdot ; Y)$ is a metric on $Y$, except that it is allowed to take infinite values.  
\item
If $X \subset \BB C$, we say that $d$ is a \emph{lower semicontinuous metric} if the function $(x,y) \rta d(x,y)$ is lower semicontinuous w.r.t.\ the Euclidean topology.  
We equip the set of lower semicontinuous metrics on $X$ with the topology on lower semicontinuous functions on $X \times X$, as in Definition~\ref{def-lsc}, and the associated Borel $\sigma$-algebra.
\end{itemize}
\end{defn}

An \emph{annular region} is a bounded open set $A\subset\BB C$ such that $A$ is homeomorphic to an open, closed, or half-open Euclidean annulus. If $A$ is an annular region, then $\bdy A$ has two connected components, one of which disconnects the other from $\infty$. We call these components the outer and inner boundaries of $A$, respectively.

\begin{defn}[Distance across and around annuli] \label{def-around-across} 
Let $d$ be a length metric on $\BB C$. 
For an annular region $A \subset\BB C$, we define $d\left(\text{across $A$}\right)$ to be the $d $-distance between the inner and outer boundaries of $A$.
We define $d \left(\text{around $A$}\right)$ to be the infimum of the $d $-lengths of a path in $A$ which disconnect the inner and outer boundaries of $A$. 
\end{defn}

\noindent
Note that both $d(\text{across $A$})$ and $d(\text{around $A$})$ are determined by the internal metric of $d$ on $A$. 
The following is a re-statement of~\cite[Definition 1.6]{pfeffer-supercritical-lqg}, with a small but important modification which we discuss after the definition. 
 
\begin{defn}[Weak LQG metric]
\label{def-metric}
Let $\mcl D'$ be the space of distributions (generalized functions) on $\BB C$, equipped with the usual weak topology.   
For $\xi > 0$, \emph{weak LQG metric with parameter $\xi$} is a measurable functions $h\mapsto D_h$ from $\mcl D'$ to the space of lower semicontinuous metrics on $\BB C$ with the following properties. Let $h$ be a \emph{GFF plus a continuous function} on $\BB C$: i.e., $h$ is a random distribution on $\BB C$ which can be coupled with a random continuous function $f$ in such a way that $h-f$ has the law of the whole-plane GFF. Then the associated metric $D_h$ satisfies the following axioms. 
\begin{enumerate}[I.]
\item \textbf{Length space.} Almost surely, $(\BB C,D_h)$ is a length space. \label{item-metric-length} 
\item \textbf{Locality.} Let $U\subset\BB C$ be a deterministic open set. 
The $D_h$-internal metric $D_h(\cdot,\cdot ; U)$ is a.s.\ given by a measurable function of $h|_U$.  \label{item-metric-local}
\item \textbf{Weyl scaling.} For a continuous function $f : \BB C \rta \BB R$, define
\eqb \label{eqn-metric-f}
(e^{\xi f} \cdot D_h) (z,w) := \inf_{P : z\rta w} \int_0^{\op{len}(P ; D_h)} e^{\xi f(P(t))} \,dt , \quad \forall z,w\in \BB C ,
\eqe 
where the infimum is over all $D_h$-continuous paths from $z$ to $w$ in $\BB C$ parametrized by $D_h$-length.
Then a.s.\ $ e^{\xi f} \cdot D_h = D_{h+f}$ for every continuous function $f: \BB C \rta \BB R$. \label{item-metric-f}
\item \textbf{Translation invariance.} For each deterministic point $z \in \BB C$, a.s.\ $D_{h(\cdot + z)} = D_h(\cdot+ z , \cdot+z)$.  \label{item-metric-translate}
\item \textbf{Tightness across scales.} Suppose that $h$ is a whole-plane GFF and let $\{h_r(z)\}_{r > 0, z\in\BB C}$ be its circle average process. 
Let $A\subset \BB C$ be a deterministic Euclidean annulus.
In the notation of Definition~\ref{def-around-across}, the random variables
\eqbn
r^{-\xi Q} e^{-\xi h_r(0)} D_h\left( \text{across $r A$} \right) \quad \text{and} \quad
r^{-\xi Q}  e^{-\xi h_r(0)} D_h\left( \text{around $r A$} \right)
\eqen
and the reciporicals of these random variables for $r>0$ are tight.   \label{item-metric-coord}  
\end{enumerate}
\end{defn}

Definition~\ref{def-metric} is the same as~\cite[Definition 1.6]{pfeffer-supercritical-lqg} except that in Axiom~\ref{item-metric-coord}, we re-scale by $r^{-\xi Q} e^{-\xi h_r(0)}$ whereas in~\cite{pfeffer-supercritical-lqg} the analogous scaling factor is $\frk c_r^{-1} e^{-\xi h_r(0)}$ for a non-explicit collection of \emph{scaling constants} $\{\frk c_r\}_{r>0}$. It was shown in~\cite[Theorem 1.9]{dg-polylog} that one can always take $\frk c_r = r^{\xi Q}$, so our definition is equivalent to the one in~\cite{pfeffer-supercritical-lqg}. 
The fact that one can take $\frk c_r = r^{\xi Q}$ is crucial for the proofs of our main theorems, since the estimates for $\frk c_r$ available in~\cite{pfeffer-supercritical-lqg} are not sufficiently precise to rule out singular points for $\xi =\xi_\crit$. 

The following is a re-statement of~\cite[Theorem 1.7]{pfeffer-supercritical-lqg}, which in turn is proven building on the tightness result in~\cite{dg-supercritical-lfpp}. 

\begin{thm}[\!\!\cite{pfeffer-supercritical-lqg}] \label{thm-lfpp-axioms}
Let $\xi > 0$. For every sequence of $\ep$'s tending to zero, there is a weak LQG metric $D$ with parameter $\xi$ and a subsequence $\{\ep_n\}_{n\in\BB N}$ for which the following is true. Let $h$ be a whole-plane GFF, or more generally a whole-plane GFF plus a bounded continuous function. Then the re-scaled LFPP metrics $\frk a_{\ep_n}^{-1} D_h^{\ep_n}$, as defined in~\eqref{eqn-lfpp} and~\eqref{eqn-gff-constant}, converge in probability to $D_h$ w.r.t.\ the metric on lower semicontinuous functions on $\BB C\times \BB C$. 
\end{thm}

It is shown in~\cite{gm-uniqueness} that for $\xi < \xi_\crit$, there is a unique weak LQG metric (up to multiplication by a deterministic positive constant).
Moreover, for $\xi  <\xi_\crit$, every weak LQG metric is a \emph{strong LQG metric}, meaning that it satisfies the following stronger form of Axiom~\ref{item-metric-coord}: for each $r > 0$, a.s.\
\eqb \label{eqn-metric-coord}
D_{h(r\cdot) + Q\log r}(z,w) = D_h(rz,rw),\quad\forall z,w\in\BB C . 
\eqe
We expect that similar statements are true for $\xi \geq \xi_\crit$, but such statements have not yet been proven.

\subsection{Main result}
\label{sec-main}

The main result of this paper is the following continuity statement for weak LQG metric at criticality.

\begin{thm} \label{thm-holder}
Let $h$ be the whole-plane GFF and let $D_h$ be a weak LQG metric with parameter $\xi=\xi_\crit$.
Let $U\subset\BB C$ be a bounded open set and let $\theta \in (0 ,   \xi_\crit /4)$.
Almost surely, there exists a random $C  \in (0,\infty)$ such that for each $z , w \in U$,
\eqb \label{eqn-holder}
D_h(z,w) \leq C \left(\max\left\{1 , \log \frac{1}{|z-w|} \right\} \right)^{-\theta} .
\eqe 
In particular, a.s.\ $D_h$ induces the Euclidean topology. 
\end{thm}

Theorem~\ref{thm-holder} is optimal in the sense that one cannot get a better modulus of continuity than a power of $\log \frac{1}{|z-w|}$ in~\eqref{eqn-holder}, as the following proposition demonstrates. See Remark~\ref{remark-optimal-exponent} for some discussion on the optimal power of $\log\frac{1}{|z-w|}$. 

\begin{prop} \label{prop-optimality}
Let $h$ be the whole-plane GFF and let $D_h$ be a weak LQG metric with parameter $\xi_\crit $. 
Let $U\subset\BB C$ be a bounded open set and let $\theta' > 3\xi_\crit/4$.
Almost surely, for every $\ep > 0$ there exist points $z,w\in U$ such that $0 < |z-w| \leq \ep$ and
\eqb \label{eqn-optimality}
D_h(z,w) \geq \left(\log \frac{1}{|z-w|}\right)^{-\theta'} .
\eqe
\end{prop}

Theorem~\ref{thm-holder} and Proposition~\ref{prop-optimality} should be contrasted with~\cite[Theorem 1.7]{lqg-metric-estimates}, which says that in the subcritical case $\xi < \xi_\crit$, the identity mapping from $\BB C$, equipped with the Euclidean metric, to $(\BB C , D_h)$ is $\chi$-H\"older continuous for any $\chi \in (0,\xi(Q-2))$. 
Our results show that for $\xi =\xi_\crit$, this map is continuous but not H\"older continuous. 
Theorem~\ref{thm-holder} should also be contrasted with~\cite[Proposition 1.11]{pfeffer-supercritical-lqg}, which implies that in the supercritical case $\xi  >\xi_\crit$, the identity mapping from $\BB C$, equipped with the Euclidean metric, to $(\BB C , D_h)$ is not continuous at any point. 

The results of this paper are similar in spirit to those of the recent work~\cite{kms-regularity}, which computes the optimal modulus of continuity for the SLE$_4$ uniformizing map and the SLE$_8$ trace (using LQG techniques). As in this paper, the optimal modulus of continuity for both situations considered in~\cite{kms-regularity} is a power of $\log(1/|\cdot|)$, whereas for other values of $\kappa$ one has local H\"older continuity.

\begin{remark}[Optimal modulus of continuity] \label{remark-optimal-exponent}
Let $\{h_\ep\}_{\ep > 0}$ be the circle average process for $h$ and let $\alpha_* $ be the supremum of the values of $\alpha  > 0$ such that the following is true. Almost surely, for each bounded open set $U\subset\BB C$, there exists a random $C>1$ such that 
\eqb \label{eqn-optimal-exponent}
\max_{z\in (e^{-n-100}\BB Z^2) \cap U} h_{e^{-n}}(z) \leq 2n - \alpha \log n + C  ,\quad\forall n\in\BB N .
\eqe
We emphasize that~\eqref{eqn-optimal-exponent} is required to hold for all $n\in\BB N$; the analogous exponent for a fixed $n$ is known to be $3/4$~\cite{drz-centered-max} (see Lemma~\ref{lem-one-scale-max}). 

Our proof of Theorem~\ref{thm-holder} shows that the theorem statement is true for any $\theta \in ( 0 , \alpha_* \xi_\crit)$. Likewise, our proof of Proposition~\ref{prop-optimality} shows that the proposition statement is true for any $\theta' > \alpha_* \xi_\crit$. 
Hence, computing the optimal modulus of continuity for $D_h$ is equivalent to computing $\alpha_*$.  
The reason why we see the quantities $\xi_\crit/4$ and $3\xi_\crit/4$ appear in Theorem~\ref{thm-holder} and Proposition~\ref{prop-optimality}, respectively, is that we can show that $1/4 \leq \alpha_* \leq 3/4$. This is done using a bound for the maximum of a centered Gaussian field from~\cite{drz-centered-max}, see Propositions~\ref{prop-gff-max} and~\ref{prop-gff-tight}. 

It was pointed out to us by Hui He that a quantity similar to $\alpha_*$ is computed for a supercritical branching random walk in~\cite[Theorem 1.2]{hz-minimal-position}.   
\end{remark}

\begin{remark}[Convergence of critical LFPP] \label{remark-lfpp-uniform}
The results of this paper do \emph{not} show that the re-scaled LFPP metrics $\{\frk a_\ep^{-1} D_h^\ep\}_{\ep > 0}$ for $\xi = \xi_\crit$ are tight with respect to the topology of uniform convergence on compact subsets of $\BB C\times\BB C$. 
Indeed, it is possible for a sequence of continuous functions to converge to a continuous function with respect to the topology on lower semicontinuous functions (Definition~\ref{def-lsc}) without converging with respect to the local uniform topology. 
A major obstacle to showing the local uniform convergence of critical LFPP is that we only have estimates for $\frk a_\ep$ which are sharp up to polylogarithmic multiplicative factors~\cite[Theorem 1.11]{dg-polylog}. This means that the error coming from our estimate for $\frk a_\ep$ could be bigger than the second-order correction in our estimate for the maximum of the circle average process of the GFF in Proposition~\ref{prop-gff-max}. 
\end{remark}

\subsection{Basic notation}
\label{sec-notation}

\noindent
We write $\BB N = \{1,2,3,\dots\}$ and $\BB N_0 = \BB N \cup \{0\}$.
\medskip

\noindent
For $a < b$, we define the discrete interval $[a,b]_{\BB Z}:= [a,b]\cap\BB Z$.
\medskip

\noindent
If $f  :(0,\infty) \rta \BB R$ and $g : (0,\infty) \rta (0,\infty)$, we say that $f(\ep) = O_\ep(g(\ep))$ (resp.\ $f(\ep) = o_\ep(g(\ep))$) as $\ep\rta 0$ if $f(\ep)/g(\ep)$ remains bounded (resp.\ tends to zero) as $\ep\rta 0$. We similarly define $O(\cdot)$ and $o(\cdot)$ errors as a parameter goes to infinity.
\medskip

\noindent
If $f,g : (0,\infty) \rta [0,\infty)$, we say that $f(\ep) \preceq g(\ep)$ if there is a constant $C>0$ (independent from $\ep$ and possibly from other parameters of interest) such that $f(\ep) \leq  C g(\ep)$. We write $f(\ep) \asymp g(\ep)$ if $f(\ep) \preceq g(\ep)$ and $g(\ep) \preceq f(\ep)$.
\medskip

\noindent
For a set $A\subset\BB C$ and $r>0$, we write
\eqbn
B_A(r) := \{z\in \BB C : \text{Euclidean distance from $z$ to $A$}  < r\} 
\eqen
For $z\in\BB C$ we write $B_z(r) = B_{\{z\}}(r)  $ for the open Euclidean ball of radius $r$ centered at $z$. 
\medskip

\noindent
For $z\in\BB C$ and $0 < a < b$, we write
\eqb \label{eqn-annulus-def}
\BB A_z(a,b) := B_z(b) \setminus \ol{B_z(a)}
\eqe
for the open annulus centered at $z$ with inradius $a$ and outradius $b$. 
\medskip

\subsection{Outline}
\label{sec-outline}

We now explain the main ideas in the proof of Theorem~\ref{thm-holder}. The proof is based on two key estimates, which are proven in Section~\ref{sec-prelim}. The first (Lemma~\ref{lem-dist-tail}) is a tail bound for $D_h$-distances which implies that for any fixed Euclidean annulus $A$, we have (in the notation of Definition~\ref{def-around-across}),
\eqb \label{eqn-outline-dist}
\BB P\left[ D_h\left(\text{around $r A + z$} \right)  > S r^{2\xi_\crit} e^{\xi_\crit h_r(z)} \right] \leq c_0 \exp\left( - c_1 \frac{(\log S)^2}{ \log\log S  } \right) ,
\eqe 
where $c_0,c_1 > 0$ are constants depending on $A$. Moreover, the same is true with ``across $rA+z$" instead of ``around $r A+z$". 

The other estimate is a tail bound for the maximum of the circle average process (Proposition~\ref{prop-gff-max}), which is a consequence of results from~\cite{drz-centered-max}. 
This estimate says that for any bounded open set $U\subset \BB C$ and any $\alpha \in (0,1/4)$, a.s.\ there is a random $C>0$ such that
\eqb \label{eqn-outline-gff-max}
|h_{e^{-n}}(z)| \leq 2 n - \alpha  \log n + C ,\quad \forall n\in\BB N, \quad \forall z \in (e^{-n-100} \BB Z^2)  \cap U.
\eqe 
\medskip

\noindent\textbf{Naive argument.}
We will first describe a naive attempt to deduce Theorem~\ref{thm-holder} from~\eqref{eqn-outline-dist} and~\eqref{eqn-outline-gff-max}, then explain why the naive argument does not work, and what modifications are needed to make it work. Fix a bounded open set $U\subset\BB C$. By~\eqref{eqn-outline-dist} applied with $S$ slightly larger than $\exp(n^{1/2})$ and a union bound over all $z\in (e^{-n-100} \BB Z^2)\cap U$, we find that with high probability
\allb \label{eqn-outline-across-around}
&\max\left\{ D_h\left(\text{around $\BB A_z(e^{-n-1} , e^{-n})$}\right) , D_h\left(\text{across $\BB A_z(e^{-n-2} , e^{-n})$}\right)\right\} \notag\\
&\qquad\qquad \leq \exp\left( \xi_\crit h_{e^{-n}}(z) - 2\xi_\crit n  +  n^{1/2 + o_n(1)} \right)       , \quad \forall z\in (e^{-n-100} \BB Z^2)\cap U .
\alle
Using the Borel-Cantelli lemma, we get that a.s.~\eqref{eqn-outline-across-around} holds for each large enough $n\in\BB N$. Equivalently,~\eqref{eqn-outline-across-around} holds a.s.\ for all $n\in\BB N$ but with a random multiplicative constant $C \in (0,\infty)$ (which does not depend on $n$ or $z$) on the right side. By slightly modifying the radii of our annuli and using a continuity estimate for the circle average process, we can arrange that~\eqref{eqn-outline-across-around} holds for all $z\in U$, not just for $z\in (e^{-n-100} \BB Z^2) \cap U$. In other words, a.s.\ there is a random $C \in (0,\infty)$ such that  
\allb \label{eqn-outline-across-around'}
&\max\left\{ D_h\left(\text{around $\BB A_z(e^{-n-1} , e^{-n})$}\right) , D_h\left(\text{across $\BB A_z(e^{-n-2} , e^{-n})$}\right)\right\} \notag\\
&\qquad\qquad \leq C \exp\left( \xi_\crit h_{e^{-n}}(z) - 2\xi_\crit n  +  n^{1/2 + o_n(1)} \right)       , \quad \forall n \in\BB N,\quad \forall z \in  U .
\alle

Now let $m,n\in\BB N$ with $m < n$. By stringing together paths in the annuli $\BB A_{z }(e^{-k-1} , e^{-k})$ and $\BB A_{z }(e^{-k-2} , e^{-k})$ for $k=0,\dots,n$, we infer from~\eqref{eqn-outline-across-around} that for any $z\in U$, 
\eqb \label{eqn-outline-sum}
D_h\left(\bdy B_z(e^{-n} ) , \bdy B_z(e^{-m} ) \right) \leq C \sum_{k=m}^n \exp\left( \xi_\crit h_{e^{-k}}(z) - 2\xi_\crit n  +  k^{1/2 + o_k(1)} \right) .
\eqe
Using continuity estimates for the circle average process, we can replace the sum by an integral on the right side of~\eqref{eqn-outline-sum}. Moreover, since our metric is lower semicontinuous, the limit of the left side of~\eqref{eqn-outline-sum} as $n\rta\infty$ provides an upper bound for $D_h(z , \bdy B_z(e^{-m})$. We therefore arrive at
\eqb \label{eqn-outline-int}
D_h\left(z , \bdy B_z(e^{-m} ) \right) \leq C \int_m^\infty \exp\left( \xi_\crit h_{e^{-t}}(z) - 2\xi_\crit t  +  t^{1/2 + o_t(1)} \right)       ,\quad\forall z \in U . 
\eqe

We attempt to estimate the right side of~\eqref{eqn-outline-int} using~\eqref{eqn-outline-gff-max}. 
This gives that for $\alpha \in (0,1/4)$, 
\eqb \label{eqn-outline-naive}
D_h\left(z , \bdy B_z(e^{-m} ) \right) \leq C \int_m^\infty \exp\left(  - \alpha \xi_\crit \log t  +  t^{1/2 + o_t(1)} \right)       ,\quad\forall z \in U . 
\eqe
This integral is plainly divergent since $t^{1/2}$ grows faster than $\log t$. In fact, even if the $t^{1/2+o_t(1)}$ error were not present, the integral would still be divergent since $\alpha \xi_\crit < 1$. 
\medskip

\noindent\textbf{Modifications to fix the argument.}
In order to make a version of the above argument work, we need two main modifications. First, we cannot just naively take a union bound over $z\in (e^{-n-100}\BB Z^2)\cap U$ when we apply~\eqref{eqn-outline-dist} to get~\eqref{eqn-outline-across-around}, since an $n^{1/2+o_n(1)}$ error inside the exponential is too big for our purposes. Instead, we will use the independence properties of the GFF to get a version of~\eqref{eqn-outline-dist} when we \emph{condition} on (roughly speaking) the circle average $h_r(z)$ (Lemma~\ref{lem-cond-reg}). This will allow us to take a union bound in a more careful manner, where we allow for a smaller error for points $z\in (e^{-n-100} \BB Z^2)\cap U$ for which $  h_{e^{-n}}(z)$ is close to $2n$. More precisely, instead of an error of $n^{1/2+o_n(1)}$ in our analog of~\eqref{eqn-outline-across-around'} we will get an error of $|2n - h_{e^{-n}}(z)|^{1/2+o_n(1)}$ (Lemma~\ref{lem-all-scales}). This results in an error of $|2t - h_{e^{-t}}(z)|^{1/2+o_t(1)}$ instead of $t^{1/2+o_t(1)}$ in~\eqref{eqn-outline-int}. 

Second, we need more refined control on the maximum of the circle average process than what we get from~\eqref{eqn-outline-gff-max} since we have $\alpha\xi_\crit < 1$. Roughly speaking, we will consider $\beta > 1/\xi_\crit$ and show that for every point $z\in U$, the set of times $t$ for which $h_{e^{-t}}(z)  \geq 2t - \beta \log t$ is small. 
A key tool to show this is an elementary estimate which says that a Brownian motion on $[0,T]$ (actually, for technical reasons we will work with a Brownian bridge) cannot spend very much time above $2t - \beta \log t$ during the interval $[T/2,T]$ if it is constrained to stay below $2t-\alpha\log t$ (Lemma~\ref{lem-bridge-bound}). Intuitively, the reason for this is that each time the Brownian bridge gets above $2t-\beta\log t$ it has a positive chance to get above $2t-\alpha\log t$ in the next $(\log t)^2$ units of time. 

Since $t\mapsto h_{e^{-t}}(z) - h_1(z)$ is a standard linear Brownian motion (see~\cite[Section 3.1]{shef-kpz}), we can apply the above Brownian motion estimate together with~\eqref{eqn-outline-gff-max} to control the amount of time that $h_{e^{-t}}(z)$ spends above $2t-\beta \log t$. This will allow us to bound the integral in~\eqref{eqn-outline-int}, with the $t^{1/2+o_t(1)}$ replaced by the aforementioned smaller error term. 

There are several technicalities involved in the above argument which we gloss over here. In particular, the tail bound in our Brownian motion estimate is not good enough to take a union bound naively, so we will need to break points up based on the value of $2t - h_{e^{-t}}(z)$, similarly to what we did to improve the $t^{1/2+o_t(1)}$ error. We will also need to consider doubly exponential scales since our Brownian motion estimate only works for times in $[T/2,T]$. 
\medskip

\noindent\textbf{Outline of Section~\ref{sec-continuity}.}
The core part of our proof, in which we make the above ideas rigorous, is given in Section~\ref{sec-continuity}. We start in Section~\ref{sec-continuity} by proving the aforementioned refined version of~\eqref{eqn-outline-across-around'}. We will also simultaneously prove a continuity estimate for the circle average process, which will be needed when we shift the centers and radii of our annuli.

In Section~\ref{sec-dist-to-int}, we explain how to string together paths to pass from a sharpened version of~\eqref{eqn-outline-across-around'} to a sharpened version of~\eqref{eqn-outline-sum}. 
The arguments up to this point all work for any $\xi  > 0$, not just $\xi  = \xi_\crit$. 
In Section~\ref{sec-dist-to-int'}, we state some simplified versions of our estimates which are specific to $\xi = \xi_\crit$. 

In Section~\ref{sec-good-bad}, we explain the aforementioned argument where we bound how much time $h_{e^{-t}}(z)$ can spend above $2t-\beta\log t$, and thereby bound the sum appearing in~\eqref{eqn-outline-sum}. In Section~\ref{sec-holder}, we conclude the proof of Theorem~\ref{thm-holder}. In Section~\ref{sec-optimality}, we prove Proposition~\ref{prop-optimality} (the proof uses a small subset of the estimates involved in the proof of Theorem~\ref{thm-holder}).

\section{Preliminaries}
\label{sec-prelim}

\subsection{Tail estimate for LQG distances}
\label{sec-dist-tail}

We will need the following concentration bound for LQG distances between sets.

\begin{lem} \label{lem-dist-tail}
Let $\xi > 0$, let $h$ be the whole-plane GFF, and let $D_h$ be a weak LQG metric. Let $U\subset \BB C$ be a connected open set and let $K_1,K_2\subset U$ be disjoint compact connected sets which are not singletons.
There are constants $c_0,c_1 > 0$ depending on $U,K_1,K_2$ and the law of $D_h$, such that the following is true. 
For each $r > 0$ and each $S >3$, 
\eqb \label{eqn-dist-tail-lower}
\BB P\left[ D_h(rK_1, r K_2  ) < S^{-1} e^{\xi h_r(0)} r^{\xi Q} \right] \leq c_0 e^{-c_1 (\log S)^2} 
\eqe 
and
\eqb \label{eqn-dist-tail-upper}
\BB P\left[ D_h(rK_1, r K_2 ; r U) > S e^{\xi h_r(0)} r^{\xi Q} \right] \leq c_0 e^{-c_1 (\log S)^2 /  \log\log S }  .
\eqe
\end{lem}

Lemma~\ref{lem-dist-tail} is similar to~\cite[Proposition 1.8]{pfeffer-supercritical-lqg}, but the latter proposition only gives a superpolynomial tail bound, rather than a lognormal tail bound. We need the lognormal tail bound since we will be working with $\xi = \xi_\crit$, so second-order corrections are important and our estimates need to be sharper than in the case when $\xi\not=\xi_\crit$. 

It is easy to see from Lemma~\ref{lem-dist-tail} that if $A$ is a fixed Euclidean annulus, then the bounds~\eqref{eqn-dist-tail-lower} and~\eqref{eqn-dist-tail-upper} also hold with $D_h(\text{across $r A$})$ and $D_h(\text{around $r A$})$ in place of $D_h(rK_1,rK_2)$ and  $D_h(rK_1,r K_2 ; rU)$ (recall Definition~\ref{def-around-across}). 

Lemma~\ref{lem-dist-tail} can be proven directly from Definition~\ref{def-metric} using a percolation-style argument based on the white noise decomposition of the GFF. But, for convenience we will instead deduce the lemma from the following LFPP estimate, which is an easy consequence of results from~\cite{dg-supercritical-lfpp,dg-polylog}. 

\begin{lem} \label{lem-dist-ep}
Let $\{\frk a_\ep^{-1} D_h^\ep\}_{\ep > 0}$ be the LFPP metrics as in~\eqref{eqn-lfpp} and~\eqref{eqn-gff-constant}. 
Let $U\subset \BB C$ be a connected open set and let $K_1,K_2\subset U$ be disjoint compact connected sets which are not singletons.
There are constants $c_0,c_1 > 0$ depending on $U,K_1,K_2$ such that the following is true. 
For each $r > 0$, each $\ep \in (0,r]$, and each $S>3$, 
\eqb \label{eqn-dist-ep-lower'}
\BB P\left[ \frk a_\ep^{-1} D_h^\ep(rK_1, r K_2  ) < S^{-1}  r^{\xi Q}   e^{\xi h_r(0)}   \right] \leq c_0 e^{-c_1 (\log S)^2} 
\eqe 
and
\eqb \label{eqn-dist-ep-upper'}
\BB P\left[ \frk a_\ep^{-1} D_h^\ep(rK_1, r K_2 ; r U) > S r^{\xi Q} e^{\xi h_r(0)}     \right] \leq c_0 e^{-c_1 (\log S)^2 /  \log\log S }  .
\eqe   
\end{lem}
\begin{proof}
By~\cite[Lemma 4.11]{dg-supercritical-lfpp}, there are constants $c_0 , c_1 > 0$ as in the lemma statement such that for each $r > 0$, each $\ep \in (0,1)$, and each $S>3$, 
\eqb \label{eqn-dist-ep-lower}
\BB P\left[ \frk a_\ep^{-1} D_h^\ep(rK_1, r K_2  ) < S^{-1}   e^{\xi h_r(0)} \frac{r \frk a_{\ep/r}}{\frk a_\ep}   \right] \leq c_0 e^{-c_1 (\log S)^2} 
\eqe 
and
\eqb \label{eqn-dist-ep-upper}
\BB P\left[ \frk a_\ep^{-1} D_h^\ep(rK_1, r K_2 ; r U) > S e^{\xi h_r(0)}  \frac{r \frk a_{\ep/r}}{\frk a_\ep} \right] \leq c_0 e^{-c_1 (\log S)^2 /  \log\log S }  .
\eqe 
By~\cite[Lemma 3.6]{dg-polylog} (with $\frk c_r = r^{\xi Q}$), there is a constant $C>1$ such that for each small enough $\ep > 0$ (depending on $r$),
\eqb \label{eqn-dist-ratio}
C^{-1} \leq \frac{r \frk a_{\ep/r}}{r^{\xi Q} \frk a_\ep} \leq C .
\eqe 
Plugging~\eqref{eqn-dist-ratio} into~\eqref{eqn-dist-ep-lower} and~\eqref{eqn-dist-ep-upper} gives~\eqref{eqn-dist-ep-lower'} and~\eqref{eqn-dist-ep-upper'}, after possibly increasing $c_0$ and/or decreasing $c_1$.
\end{proof}

We now take a limit as $\ep\rta 0$ in Lemma~\ref{lem-dist-ep} to get the following lemma.

\begin{lem} \label{lem-dist-tail-ssl}
The statement of Lemma~\ref{lem-dist-tail} holds if $D_h$ is a subsequential limit of the re-scaled LFPP metrics $\{\frk a_\ep^{-1} D_h^\ep\}$. 
\end{lem}

We note that Lemma~\ref{lem-dist-tail-ssl} does not immediately imply Lemma~\ref{lem-dist-tail} since there could in principle be metrics satisfying the axioms of Definition~\ref{def-metric} which do not arise as subsequential limits of LFPP.

\begin{proof}[Proof of Lemma~\ref{lem-dist-tail-ssl}]
By assumption, there is a sequence $\mcl E$ of $\ep$-values tending to zero such that $\frk a_\ep^{-1} D_h^\ep \rta D_h$ in law along $\mcl E$ w.r.t.\ the topology of Definition~\ref{def-lsc}. By the Skorokhod representation theorem, we can couple the metrics $\{\frk a_\ep^{-1} D_h^\ep\}_{\ep > 0}$ with $D_h$ so that the convergence occurs a.s. (note that in this coupling the metrics are not necessarily all defined w.r.t.\ the same GFF instance). 

We first deduce~\eqref{eqn-dist-tail-lower} from~\eqref{eqn-dist-ep-lower'}. 
Let  $K_1' , K_2' \subset U$ be disjoint connected compact sets such that for each $i\in \{1,2\}$, the set $K_i$ is contained in the interior of $K_i'$. By Definition~\ref{def-lsc}, for each $u \in r K_1$ and each $v\in r K_2$, there exists a sequence of pairs of points $(u_\ep , v_\ep)_{\ep\in\mcl E}$ such that along $\mcl E$, we have the convergence $|u_\ep - u| \rta 0$, $|v_\ep - v| \rta 0$, and $\frk a_\ep^{-1} D_h^\ep(u_\ep ,v_\ep) \rta D_h(u,v)$. 
For each small enough $\ep >0$, we have $u_\ep \in r K_1'$ and $v_\ep \in r K_2'$. Therefore, a.s.\ 
\eqbn
D_h(u,v) \geq \limsup_{\ep\rta 0} \frk a_\ep^{-1} D_h^\ep(rK_1' , r K_2'  ) ,\quad \forall u \in r K_1, \quad \forall v \in r K_2 .
\eqen
From this and~\eqref{eqn-dist-ep-lower'} with $K_1',K_2'$ in place of $K_1,K_2$, we obtain~\eqref{eqn-dist-tail-lower}. 

We next deduce~\eqref{eqn-dist-tail-upper} from~\eqref{eqn-dist-ep-upper'}. The proof is slightly more involved than one might initially expect since we do not know that $\frk a_\ep^{-1} D_h^\ep(\cdot,\cdot ; r U)  \rta D_h(\cdot,\cdot ; rU)$ with respect to the metric on lower semicontinuous functions, so one needs to find pairs of points $u,v$ for which $D_h(u,v;U) = D_h(u,v)$.  

Let $U'\subset U$ be a an open set such that $\ol U'$ is a compact subset of $U$ and $K_1,K_2\subset U'$.
By~\eqref{eqn-dist-ep-upper'} with $U'$ in place of $U$, we can find constants $c_0,c_1 > 0$ as in the lemma statement such that for each $S>3$, it holds with probability at least $1-c_0 e^{-c_1 (\log S)^2 / \log\log S}$ that there is a subsequence $\mcl E'\subset \mcl E$ such that 
\eqbn
\frk a_\ep^{-1} D_h^\ep(rK_1, r K_2 ; r U') \leq S e^{\xi h_r(0)}  r^{\xi Q} ,\quad\forall \ep \in \mcl E'.
\eqen
Henceforth assume that such a subsequence $\mcl E'$ exists, which happens with probability at least $1-c_0 e^{-c_1 (\log S)^2 / \log\log S}$. We will establish an upper bound for $D_h(rK_1, r K_2 ; r U)$.

For $\ep\in\mcl E'$, let $P^\ep : [0,T^\ep] \rta r U'$ be a path in $r U'$ from $r K_1$ to $r K_2$ with $\frk a_\ep^{-1} D_h^\ep$-length $T^\ep \leq 2 S e^{\xi h_r(0)} r^{\xi Q}  $, parametrized by its $\frk a_\ep^{-1} D_h^\ep$-length. We extend the definition of $P^\ep$ to $[0,\infty)$ by setting $P^\ep(t) = P^\ep(T^\ep)$ for $t\geq T^\ep$. 

Almost surely, we have $D_h(r U' , r \bdy U) > 0$, so a.s.\ we can find a partition $0  =t_0 < \dots < t_N = 2 S r^{\xi Q} e^{\xi h_r(0)}$ such that 
\eqb \label{eqn-dist-tail-partition}
\sup_{n\in [1,N]_{\BB Z}} (t_n - t_{n-1}) \leq \frac12 D_h(r U' , r \bdy U) . 
\eqe 
Since $r \ol U'$ is compact, we can a.s.\ find a subsequence $\mcl E'' \subset \mcl E'$ and points $u_n \in r \ol U'$ for $n\in[0,N]_{\BB Z}$ such that for each $n\in[0,N]_{\BB Z}$, we have $|P^\ep(t_n) - u_n| \rta 0$ as $\ep\rta 0$ along $\mcl E''$. Then $u_0 \in r K_1$, $u_N \in r K_2$, and by Definition~\ref{def-lsc},  
\eqb \label{eqn-dist-tail-inc}
D_h(u_{n-1} , u_n) 
\leq \liminf_{\mcl E'' \ni \ep\rta 0} \frk a_\ep^{-1} D_h^\ep(P^\ep(t_{n-1}) , P^\ep(t_n))  
\leq  t_n - t_{n-1}   .
\eqe 
By~\eqref{eqn-dist-tail-partition} and~\eqref{eqn-dist-tail-inc}, for each $n \in [1,N]_{\BB Z}$, each path from $u_{n-1}$ to $u_n$ of near-minimal $D_h$-length is contained in $r U$, so $D_h(u_{n-1} , u_n) = D_h(u_{n-1} , u_n ; r U)$. 
By this, the triangle inequality, and~\eqref{eqn-dist-tail-inc}, 
\eqbn
D_h(rK_1, r K_2 ; r U)  \leq \sum_{n=1}^N D_h(u_{n-1} , u_n ; r U) \leq \sum_{n=1}^N (t_n - t_{n-1}) = t_N - t_0 = 2 S r^{\xi Q} e^{\xi h_r(0)} . 
\eqen
This gives~\eqref{eqn-dist-tail-upper} with $2S$ in place of $S$, which is sufficient. 
\end{proof}

\begin{proof}[Proof of Lemma~\ref{lem-dist-tail}]
Suppose that $D_h$ is an arbitrary weak LQG metric and let $\wt D_h$ be a weak LQG metric which is a subsequential limit of LFPP (such a weak LQG metric exists by Theorem~\ref{thm-lfpp-axioms}). 
By~\cite[Theorem 1.10]{dg-polylog}, there is a deterministic constant $C>1$ such that a.s.\ 
\eqbn
C^{-1} \wt D_h(z,w) \leq D_h(z,w) \leq C \wt D_h(z,w) ,\quad \forall z,w\in\BB C .
\eqen
Combining this with Lemma~\ref{lem-dist-tail-ssl} concludes the proof. 
\end{proof}

\subsection{Maximum of the GFF}
\label{sec-gff-max}

A key input in our proof of Theorem~\ref{thm-holder} is the following estimate for the maximum of the GFF.

\begin{prop} \label{prop-gff-max} 
Let $U\subset \BB C$ be a bounded open set, let $\alpha \in (0,1/4)$, and let $k\in\BB N$. Almost surely, there is a random $C>0$ such that
\eqb \label{eqn-gff-max}
|h_{e^{-n}}(z)| \leq 2 n - \alpha  \log n + C ,\quad \forall n\in\BB N, \quad \forall z \in (e^{-n-k} \BB Z^2)  \cap U.
\eqe 
\end{prop}

For the proof of Proposition~\ref{prop-optimality}, we also need a lower bound for the maximum of the circle average process.

\begin{prop} \label{prop-gff-tight} 
Let $U\subset \BB C$ be a bounded open and let $k\in\BB N$. Almost surely, the random variables
\eqb \label{eqn-gff-tight}
\max_{z\in ( e^{-n-k}\BB Z^2)\cap U} h_{e^{-n}}(z) - \left( 2 n - \frac34 \log n \right) 
\eqe 
for $n\in\BB N$ are tight. In particular, for each $\alpha' > 3/4$, a.s.\ there exist infinitely many values of $n\in\BB N$ such that
\eqb \label{eqn-gff-lower}
\max_{z\in (e^{-n-k}\BB Z^2)\cap U} h_{e^{-n}}(z)  \geq 2n - \alpha' \log n. 
\eqe
\end{prop}

We will deduce Propositions~\ref{prop-gff-max} and~\ref{prop-gff-tight} from the following estimate for a single value of $n$, with a zero-boundary GFF instead of a whole-plane GFF. 
 
\begin{lem} \label{lem-one-scale-max}
Let $U\subset\BB C$ be a bounded open set, let $V\subset\BB C$ be a bounded, simply connected open set which contains $\ol U$, and let $\rng h$ be the zero-boundary GFF on $V$.
For each $n,k\in\BB N$ and each $S>1$, 
\eqb \label{eqn-one-scale-max}
\BB P\left[ \max_{z \in   (e^{-n - k}  \BB Z^2  ) \cap U } |\rng h_{e^{-n}}(z)| > 2 n - \frac34  \log n +S \right] \preceq   e^{2k} S e^{-2 S} ,
\eqe
with the implicit constant depending only on $U,V$. Furthermore, the random variables
\eqb \label{eqn-one-scale-tight}
\max_{z\in ( e^{-n-k}\BB Z^2)\cap U} \rng h_{e^{-n}}(z) - \left( 2 n - \frac34 \log n \right) 
\eqe
for $n\in\BB N$ are tight. 
\end{lem}

The estimate~\eqref{eqn-one-scale-max} from Lemma~\ref{lem-one-scale-max} is a straightforward consequence of the following general result on the maximum of centered Gaussian fields, which is~\cite[Proposition 1.1]{drz-centered-max}.

\begin{prop}[\!\!\cite{drz-centered-max}] \label{prop-drz}
Let $U\subset\BB C$ be a bounded open set, let $N \geq 1$, and let $U_N = (N U) \cap \BB Z^2$. 
Let $\psi_N : U_N\rta\BB R$ be a centered Gaussian process. Assume that there exists a constant $c_0 > 0$ such that for all $u,v\in U_N$, $u\not= v$, 
\eqb \label{eqn-drz-var}
\op{Var} \psi_N(u) \leq \log N + c_0 
\eqe
and
\eqb \label{eqn-drz-diff}
\BB E\left[ (\psi_N(u) - \psi_N(v))^2 \right] \leq 2 \log|u-v|  - \left| \op{Var} \psi_N(u) - \op{Var} \psi_N(v) \right|  + c_0 . 
\eqe
There is a constant $A = A(c_0,U) > 0$ such that for every $S >0$, 
\eqb \label{eqn-drz}
\BB P\left[ \max_{u\in U_N} \psi_N(u) \geq 2 \log N - \frac34 \log \log N + S \right] \leq A S e^{-2S} e^{-S^2/(A N)} .
\eqe 
\end{prop}

For the proof of the tightness statement in Lemma~\ref{lem-one-scale-max}, we will use the following result, which is~\cite[Theorem 1.2]{drz-centered-max}.

\begin{prop}[\!\!\cite{drz-centered-max}] \label{prop-drz-tight}
Let $U\subset\BB C$ be a bounded open set, let $N \geq 1$, and let $U_N = (N U) \cap \BB Z^2$. 
Let $\psi_N : U_N\rta\BB R$ be a centered Gaussian process. Assume that there exists a constant $c_0 > 0$ such that for all $u,v\in U_N$, $u\not= v$, the estimates~\eqref{eqn-drz-var} and~\eqref{eqn-drz-diff} from Proposition~\ref{prop-drz} hold, and also
\eqb \label{eqn-drz-cov}
\left|\op{Cov}(\psi_N(u) , \psi_N(v)) - (\log N - \log |u-v|) \right| \leq c_0 .
\eqe
Then 
\eqbn
\BB E\left[ \max_{u\in U_N} \psi_N(u) \right] =  2 \log N - \frac34 \log \log N + O(1),
\eqen
with the $O(1)$ depending only on $c_0$, and the random variables
\eqbn
\max_{u\in U_N} \psi_N(u)  - \left(   2 \log N - \frac34 \log \log N  \right) 
\eqen
for $N\in\BB N$ are tight. 
\end{prop}

We note that~\cite{drz-centered-max} only considers the case when $U$ is the unit square. Propositions~\ref{prop-drz} and~\ref{prop-drz-tight} in the case of a general bounded open $U$ can be deduced from the case of the unit square by covering $U$ by finitely many translated copies of the unit square.

\begin{proof}[Proof of Lemma~\ref{lem-one-scale-max}]
Let $K$ be the smallest integer such that $K \geq 2e^k$, so that 
\eqbn
  e^{-n-k} K \in  [2 e^{-n} , (2+e^{-k})e^{-n} ] .
\eqen 
Let $\phi_n$ be the centered Gaussian process on on $(\frac{e^{n+k}}{K} U) \cap \BB Z^2$ defined by 
\eqbn
\phi_n(u) = \rng h_{e^{-n}}( e^{-n-k} K u) ,\quad \forall u \in  (\frac{e^{n+k}}{K} U) \cap \BB Z^2 .
\eqen
We will check the hypotheses of Propositions~\ref{prop-drz} and~\ref{prop-drz-tight} (with $N = e^{n+k}/K \approx  e^n /2$) for the field $\phi_n$.
This will be done using the calculations from~\cite[Section 3.1]{shef-kpz}, and will lead to a version of the lemma statement with $  e^{-n-k} K \BB Z^2$ instead of $e^{-n-k}\BB Z^2$. 
We will then replace $K\BB Z^2$ by $K\BB Z^2+x$ for $x\in [0,K]_{\BB Z}^2$ and take a union bound over the $O_k(e^{ 2k})$ possibilities for $x$. 

The reason why we do not prove a bound for $(e^{-n-k}\BB Z^2)\cap U$ directly is that the formulas from~\cite{shef-kpz} have a simpler form when the circles we are taking averages over are disjoint. Note that the circles $\bdy B_z(e^{-n} )$ and $\bdy B_w(e^{-n})$ can intersect for $z,w\in e^{-n-k}\BB Z^2$. 

Throughout, we let $\op{CR}(\cdot;V)$ be the conformal radius w.r.t.\ $V$ and let $\op{Gr}_V$ be the Green's function for Brownian motion killed upon exiting $V$. 
Also, $O(1)$ denotes a quantity which is bounded above in absolute value by a constant depending only on $U,V$. 
\medskip

\noindent\textit{Step 1: verifying the hypotheses of Propositions~\ref{prop-drz} and~\ref{prop-drz-tight}.}
For $z,w\in V$ with $|z-w| \geq   e^{-n-k} K$, the disks $B_z(e^{-n} )$ and $B_w(e^{-n} )$ are disjoint. So, we can apply~\cite[Proposition 3.2]{shef-kpz} to get 
\eqb \label{eqn-circle-avg-var}
\op{Var} \rng h_{e^{-n}}(z) = n + \log \op{CR}(z;V) 
\eqe
and
\eqb \label{eqn-circle-avg-cov}
\op{Cov}\left( \rng h_{e^{-n}}(z) , \rng h_{e^{-n}}(w) \right) = \op{Gr}_V(z,w) .
\eqe
Hence
\allb \label{eqn-circle-avg-diff}
\BB E\left[\left( \rng h_{e^{-n}}(z)  - \rng h_{e^{-n}}(w) \right)^2\right] 
= 2 n + \log \op{CR}(z;V)  + \log \op{CR}(w;V)  - 2 \op{Gr}_V(z,w) .
\alle

\medskip
\noindent\textit{Proof of~\eqref{eqn-drz-var}.}
Since $\ol U\subset V$, we have that $|\log \op{CR}(z;V)|$ is bounded above by a constant on $U$. Hence, by~\eqref{eqn-circle-avg-var}, 
\eqb \label{eqn-lattice-field-var}
|\op{Var} \phi_n(u)  -  n |  = O(1) , \quad \forall u \in \left( \frac{e^{n+k}}{K} U \right) \cap \BB Z^2.
\eqe
\medskip

\noindent\textit{Proof of~\eqref{eqn-drz-diff}.}
The Green's function satisfies $\op{Gr}_V(z,w) = \log(1/|z-w|) + g$, for a function $g$ which is continuous on $V$ and hence bounded on $U$. 
By~\eqref{eqn-circle-avg-var} and our above bounds for $\op{CR}(z;V)$ and $\op{Gr}_V(z,w)$, we get that for $u,v\in  \left( \frac{e^{n+k}}{K} U \right) \cap \BB Z^2$ with $u\not=v$, 
\allb
\BB E\left[\left( \phi_n(u)  - \phi_n(v) \right)^2\right] 
 = 2 n   -  2\log \frac{1}{|e^{-n} u - e^{-n} v|}      + O(1) 
 = 2 \log |u-v| + O(1)  .
\alle
By~\eqref{eqn-lattice-field-var}, we have $|\op{Var} \phi_n(u)  - \op{Var} \phi_n(v)| = O(1)$, so  
\eqb \label{eqn-lattice-field-diff}
\BB E\left[\left( \phi_n(u)  - \phi_n(v) \right)^2\right]  \leq 2 \log |u-v|   - |\op{Var} \phi_n(u)  - \op{Var} \phi_n(v)|      + O(1) .
\eqe
\medskip

\noindent\textit{Proof of~\eqref{eqn-drz-cov}.} By~\eqref{eqn-circle-avg-cov}, for $u,v \in  \left( \frac{e^{n+k}}{K} U \right) \cap \BB Z^2$ with $u\not=v$,  
\allb \label{eqn-lattice-field-cov}
\op{Cov}\left( \phi_n(u) , \phi_n(v)  \right) 
&= \op{Gr}_V(e^{-n} u , e^{-n} v) \notag \\
&= \log \frac{1}{|e^{-n} u - e^{-n} v|} + O(1) \notag \\
&= n - \log |u-v| + O(1) .
\alle 
\medskip

\noindent\textit{Step 2: applying Propositions~\ref{prop-drz} and~\ref{prop-drz-tight}.}
By~\eqref{eqn-lattice-field-var} and~\eqref{eqn-lattice-field-diff}, we can apply~\cite[Proposition 1.1]{drz-centered-max} to the field $\phi_n$ (with $N = e^{n+k}/K \approx \frac12 e^n $) to get that for each $n\in\BB N$ and each $S>1$, 
\eqbn
\BB P\left[ \max_{u \in \left( \frac{e^{n+k}}{K} U \right) \cap \BB Z^2 } |\phi_n(u)| > 2 n - \frac34  \log n +S\right] \preceq    S e^{-2 S} 
\eqen
with an implicit constant depending only on $U,V$ (note that we ignore the factor of $ e^{-S^2/(A N)}$ in~\eqref{eqn-drz}, which is not necessary for our purposes). 
This implies that 
\eqb \label{eqn-use-drz}
\BB P\left[ \max_{z \in   (  e^{-n-k} K \BB Z^2) \cap U } |\rng h_{e^{-n}}(z)| > 2 n - \frac34  \log n +S \right] \preceq    S e^{-2 S} .
\eqe
Similarly, by~\eqref{eqn-lattice-field-var}, \eqref{eqn-lattice-field-diff}, and~\eqref{eqn-lattice-field-cov}, we can apply Proposition~\ref{prop-drz-tight} to get that the random variables
\eqb  \label{eqn-use-drz-tight}
\max_{z \in  (  e^{-n-k} K \BB Z^2) \cap U } \rng h_{e^{-n}}(z)  - \left(   2 n - \frac34 \log n  \right) 
\eqe 
are tight.
\medskip

\noindent\textit{Step 3: extending from $(Ke^{-n-k}\BB Z^2)\cap U$ to $(e^{-n-k}\BB Z^2)\cap U$.}
For any $x \in [0,K]_{\BB Z}^2$, the same argument leading to~\eqref{eqn-use-drz} and~\eqref{eqn-use-drz-tight} shows that
\eqb \label{eqn-use-drz-x}
\BB P\left[ \max_{z \in  (  e^{-n-k} (K  \BB Z^2 + x) ) \cap U } |\rng h_{e^{-n}}(z)| > 2 n - \frac34  \log n +S \right] \preceq    S e^{-2 S}  
\eqe
and the random variables 
\eqb \label{eqn-use-drz-tight-x} 
\max_{z \in  ( e^{-n-k}( K \BB Z^2+x) ) \cap U } \rng h_{e^{-n}}(z)  - \left(   2 n - \frac34 \log n  \right)  
\eqe
are tight.  
Each $z\in e^{-n-k} \BB Z^2$ belongs to $e^{-n-k}( K \BB Z^2 + x)$ for some $x \in [0,K]_{\BB Z}^2$. 
Hence, by a union bound over $O_k(e^{2k})$ possibilities for $x$, the estimate~\eqref{eqn-use-drz-x} implies~\eqref{eqn-one-scale-max}.
Similarly, the tightness of the random variables in~\eqref{eqn-use-drz-tight-x} implies the tightness of the random variables in~\eqref{eqn-one-scale-tight}. 
\end{proof}

\begin{proof}[Proof of Propositions~\ref{prop-gff-max} and~\ref{prop-gff-tight}]
Let $V\subset\BB C$ be a bounded, simply connected open set which contains $\ol U$, and let $\rng h$ be the zero-boundary GFF on $V$.
For any $\nu > 1/2$, we can apply Lemma~\ref{lem-one-scale-max} with $S = \nu \log n$, followed by the Borel-Cantelli lemma, to get that a.s.\ for each large enough $n\in\BB N$, 
\eqbn
 \max_{z \in  U \cap (e^{-n - k}  \BB Z^2  ) } |\rng h_{e^{-n}}(z)| \leq 2 n - \left(\frac34 - \nu \right)   \log n .
\eqen
By taking $\nu = 3/4-\alpha$ and choosing $C$ to be large enough to account for finitely many small values of $n$, we obtain~\eqref{eqn-gff-max} with $\rng h$ in place of $h$.
To deduce the estimate for $h$, we use the Markov property of the whole-plane GFF (see, e.g.,~\cite[Lemma 2.2]{gms-harmonic}) to write $h|_V = \rng h + \frk h$, where $\frk h$ is a random harmonic function on $V$. Since $\frk h$ is continuous and $\ol U \subset V$, we have $\sup_{z\in U} |(h-\rng h)(z)| = \sup_{z\in U} |\frk h(z)| < \infty$, so~\eqref{eqn-gff-max} for $\rng h$ implies~\eqref{eqn-gff-max} for $h$. This gives Proposition~\ref{prop-gff-max}. 

As for Proposition~\ref{prop-gff-tight}, the tightness of the random variables in~\eqref{eqn-one-scale-tight} implies the tightness of the random variables in~\eqref{eqn-gff-tight} via the same argument used for~\eqref{eqn-gff-max} above.
To deduce~\eqref{eqn-gff-lower}, we note that the tightness of~\eqref{eqn-gff-tight} implies that for any $\alpha' > 3/4$, 
\eqbn
\lim_{n\rta\infty} \BB P\left[ \max_{z \in ( e^{-n -k } \BB Z^2)\cap U} \rng h_{e^{-n}}(z)  \geq   2 n - \alpha' \log n  \right] = 1 .
\eqen
Hence a.s.\ there are infinitely values of $n\in\BB N$ for which $\max_{z \in (2 e^{-n - k } \BB Z^2)\cap U} \rng h_{e^{-n}}(z)  \geq   2 n - \alpha' \log n $.  
\end{proof}

\section{Proof of continuity}
\label{sec-continuity}

\subsection{Estimate for distances in annuli in terms of circle averages} 
\label{sec-all-scales}

Throughout this subsection and the next, we allow for a general $\xi  > 0$ and corresponding $Q = Q(\xi) > 0$, i.e., we do not require that $\xi = \xi_\crit$ and $Q=2$. 

For $z\in\BB C$ and $t\in\BB R$, we define the annuli
\eqb \label{eqn-annuli-def}
A_z^\circ (e^{-t}) := \BB A_z(e^{-t-51/100} , e^{-t-1/2})      \quad \text{and} \quad
A_z^\parallel (e^{-t}) := \BB A_z(e^{-t-100} , e^{-t})   .
\eqe 
Note that $A_z^\circ (e^{-t})$ is contained in a small neighborhood of the circle $\bdy B_z(e^{-t-1/2})$ and $A_z^\parallel (e^{-t})$ contains all of $B_z(e^{-t})$ except for a small ball centered at $z$. Our estimates for distances will be proven by stringing together paths between the inner and outer boundaries of annuli of the form $A_z^\parallel(e^{-t})$ and paths in annuli of the form $A_z^\circ(e^{-t})$ which disconnect the inner and outer boundaries. To this end, we will need to bound $D_h(\text{around $A_z^\circ(e^{-t})$})$ and $D_h(\text{across $A_z^\parallel(e^{-t})$})$. We will estimate these distances in terms of the circle average process for $h$. Since we will need to consider circles with slightly different center points and radii, we will also need a continuity estimate for this circle average process, which we prove simultaneously with our bounds for distances. To make our estimates more convenient to state, we introduce the following notation. 

\begin{defn} \label{def-triple-max}
For $z\in\BB C$ and $t \geq 0$, let $M_z(e^{-t})$ be the maximum of the following four quantities:
\begin{enumerate}
\item $\exp\left( -\xi h_{e^{-t}}(z) +  \xi Q t \right) D_h\left(\text{around $A_z^\circ(e^{-t})$}\right) $;
\item $\exp\left( -\xi h_{e^{-t}}(z) + \xi Q t \right) D_h\left(\text{across $A_z^\parallel(e^{-t})$}\right)$;
\item $\left[ \exp\left( - \xi h_{e^{-t}}(z) + \xi Q t \right)  D_h\left(\text{across $A_z^\parallel(e^{-t})$}\right) \right]^{-1}$;
\item $\sup_{r \in [ e^{-t-100}   , e^{-t} ]} \sup_{w\in B_z(e^{-t}-r )} \exp\left( |h_r(w) - h_{e^{-t}}(z)| \right)$. 
\end{enumerate}
\end{defn}

The reason why we include $\left[ \exp\left( - \xi h_{e^{-t}}(z) + \xi Q t \right)  D_h\left(\text{across $A_z^\parallel(e^{-t})$}\right) \right]^{-1}$ in our definition of $M_z(e^{-t})$ is because we will need a lower bound for $D_h$-distances for the proof of Proposition~\ref{prop-optimality}. 
The goal of this subsection is to prove the following estimate for $M_z(e^{-t})$ in terms of the circle average process for $h$.

\begin{lem} \label{lem-all-scales}
Let $U\subset\BB C$ be a bounded open set and let $\zeta \in (0,1/2)$. 
Almost surely, there exists a random $C \in (0,\infty)$ such that the random variable $M_z(e^{-n})$ from Lemma~\ref{lem-cond-reg} satisfies
\eqb  \label{eqn-all-scales}
M_z(e^{-n}) \leq C \exp\left( \left| 2n - h_{e^{-n}}(z)\right|^{1/2+\zeta}\right) ,\quad\forall n\in\BB N, \quad\forall z \in (e^{-n-100} \BB Z^2) \cap U . 
\eqe 
\end{lem} 

Proposition~\ref{prop-gff-max} implies that a.s.\ $2n - h_{e^{-n}}(z) > 0$ for each large enough $n\in\BB N$ and each $z\in (e^{-n-100} \BB Z^2) \cap U$, so the absolute value in~\eqref{eqn-all-scales} is only relevant for finitely many values of $n$. 

For most of the proof of Lemma~\ref{lem-all-scales}, we will prove estimates in terms of $h_{e^{-t}}(z) - h_1(z)$ instead of $h_{e^{-t}}(z)$. The reason why this is convenient is that, as explained in the proof of Lemma~\ref{lem-cond-reg}, $h_{e^{-t}}(z) - h_1(z)$ is independent from $M_z(e^{-t})$. Eventually, we will absorb $h_1(z)$ into a global constant using the fact that $\sup_{z\in U} |h_1(z)|$ is a.s.\ finite.

To prove Lemma~\ref{lem-all-scales}, we will first use basic estimates for $D_h$ and for $\{h_\ep\}_{\ep > 0}$ to bound the conditional probability that $M_z(e^{-t}) > S$ given $h_{e^{-t}}(z) - h_1(z)$, uniformly over all $z\in \BB C$ (Lemma~\ref{lem-cond-reg}). This will allow us to estimate 
\eqbn
\BB P\left[ M_z(e^{-t})  > \exp\left( \left| 2t - ( h_{e^{-t}}(z) - h_1(z) )\right|^{1/2+\zeta}\right) \,|\, h_{e^{-t}}(z) - h_1(z) \right] .
\eqen
By combining the resulting estimate with the Gaussian tail bound for the Gaussian random variable $h_{e^{-t}}(z) -h_1(z)$, we will get for each $k\in \BB N$ a bound for the probability that $M_z(e^{-t})  > \exp\left( k^{1/2+\zeta}\right)$ and $2t - (h_{e^{-t}}(z) - h_1(z)) \in [k, k+1]$. We will then apply this estimate together with a union bound over all $z\in (e^{-t-100}\BB Z^2) \cap U$ to prove a version of Lemma~\ref{lem-all-scales} which holds for a single value of $t>0$, rather than for all $n\in\BB N$ (Lemma~\ref{lem-max-and-circle}). Lemma~\ref{lem-all-scales} will be deduced from this and a union bound over $n$.

\begin{lem} \label{lem-cond-reg}
There are universal constants $c_0,c_1> 0$ such that for each $z\in\BB C$, $t \geq 0$, and $S > 3$, a.s.\
\eqb  \label{eqn-cond-reg}
\BB P\left[ M_z(e^{-t}) > S \,|\, h_{e^{-t}}(z) - h_1(z) \right] \leq c_0 e^{-c_1 (\log S)^2/(\log\log S)^2} .
\eqe 
\end{lem}

The reason why we can condition on $h_{e^{-t}}(z) - h_1(z)$ in Lemma~\ref{lem-cond-reg} is the following basic fact about the whole-plane GFF.

\begin{lem} \label{lem-gff-ind}
Let $h$ be a whole-plane GFF normalized so that $h_1(0) = 0$. 
For each $t \in\BB R$, the process $\{h_{e^{-s}}(0) - h_{e^{-t}}(0) : s \leq t \}$ is independent from $(h-h_{e^{-t}}(0)) |_{B_0( e^{-t} )}$. 
\end{lem}

Lemma~\ref{lem-gff-ind} has been used implicitly in several places in the literature, e.g., in the definition of the quantum cone in~\cite{wedges}.  
But, to our knowledge this fact is not explicitly stated as a lemma elsewhere, so we will give a proof. 

\begin{proof}[Proof of Lemma~\ref{lem-gff-ind}]
Let $\mcl H$ be the Hilbert space used to define $h$. That is, $\mcl H$ is the Hilbert space completion of the space of smooth functions $f$ on $\BB C$ whose average over $\bdy\BB D$ vanishes and such that $\int_{\BB C} |\nabla f(z)|^2 \,dz < \infty$, with respect to the Dirichlet inner product $(f,g)_\nabla = \int_{\BB C} \nabla f(z) \cdot \nabla g(z) \,dz$. If $\{f_j\}_{j\in\BB N}$ is an orthonormal basis for $\mcl H$, then we can write $h = \sum_{j=1}^\infty X_j f_j$, where the $X_j$'s are i.i.d.\ standard Gaussian random variables and the sum converges in the distributional sense. 

Let $\mcl H^0$ (resp.\ $\mcl H^\dagger$) be the subspace of $\mcl H$ consisting of functions which are constant on (resp.\ have mean zero on) every circle centered at 0. 
By~\cite[Lemma 4.9]{wedges}, $\mcl H$ is the orthogonal direct sum of $\mcl H^0$ and $\mcl H^\dagger$. Hence every function $f\in \mcl H$ can be expressed uniquely as the sum of a function $f^0 \in \mcl H^0$ and a function $f^\dagger \in \mcl H^\dagger$. A valid choice for $f^0$ and $f^\dagger$ (hence the only choice) is to take $f^0$ to be the function whose value on each circle centered at 0 is the average of $f$ over that circle, and to take $f^\dagger  = f - f^0$. 

We can choose our orthonormal basis for $\mcl H$ to be the disjoint union of an orthonormal basis for $\mcl H^0$ and an orthonormal basis for $\mcl H^\dagger$. 
The above description of the GFF then shows that $h = h^0 + h^\dagger$, where $h^0$ and $h^\dagger$ are independent, $h^0(z) = h_{|z|}(0)$ for each $z\in\BB C$, and $h^\dagger$ is the generalized function $h - h^0$. 

The process $s \mapsto h_{e^{-s}}(0)  $ is a standard two-sided Brownian motion~\cite[Section 3.1]{shef-kpz}. By the independent increments property of Brownian motion and the independence of $h^0$ and $h^\dagger$, it follows that $\{h_{e^{-s}}(0) - h_{e^{-t}}(0) : s \leq t \}$ is independent from the pair consisting of $h^\dagger$ and $\{h_{e^{-s}}(0) - h_{e^{-t}}(0) : s \geq t\}$. Since $(h-h_{e^{-t}}(0)) |_{B_0( e^{-t} )}$ is determined by $h^\dagger$ and $\{h_{e^{-s}}(0) - h_{e^{-t}}(0) : s \geq t\}$, we obtain the lemma statement. 
\end{proof}

\begin{proof}[Proof of Lemma~\ref{lem-cond-reg}]
By the locality and Weyl scaling properties of $D_h$ (Axioms~\ref{item-metric-local} and~\ref{item-metric-f}), the random variable $M_z(e^{-t})$ is a.s.\ determined by $h|_{B_z(e^{-t})}$, viewed modulo additive constant, so in particular it is determined by $(h  -h_{e^{-t}}(z)) |_{B_z(e^{-t})}$. 
By Lemma~\ref{lem-gff-ind} and the translation invariance of the law of $h$, viewed modulo additive constant, it follows that $M_z(e^{-t})$ is independent from $h_{e^{-t}}(z) - h_1(z)$. 

Consequently, it suffices to prove a tail bound for the unconditional law of $M_z(e^{-t})$. 
By the translation invariance of the law of $h$, viewed modulo additive constant, together with Lemma~\ref{lem-dist-tail},
\allb  \label{eqn-cond-reg-dist}
&\BB P\left[ \exp\left( -\xi h_{e^{-t}}(z) + \xi Q t \right) \max\left\{ D_h\left(\text{around $A_z^\circ(e^{-t})$}\right)  , D_h\left(\text{across $A_z^\parallel(e^{-t})$}\right) \right\} > S \right] \notag\\
&\qquad\qquad\qquad\qquad \leq a_0 e^{-a_1 (\log S)^2 / (\log\log S)^2}  
\alle
for constants $a_0,a_1 > 0$ depending only on $\xi$. Furthermore, by possibly increasing $a_0$ and decreasing $a_1$, we can arrange that also
\allb  \label{eqn-cond-reg-dist-lower}
 \BB P\left[ \exp\left( \xi h_{e^{-t}}(z) - \xi Q t \right)  D_h\left(\text{across $A_z^\parallel(e^{-t})$}\right) < S^{-1} \right]
  \leq a_0 e^{-a_1 (\log S)^2  }  .
\alle

To deal with the supremum of circle averages involved in the definition of $M_z(e^{-t})$, we first use the scale and translation invariance of the law of $h$, viewed modulo additive constant, to get
\eqb  \label{eqn-circle-max-law}
\sup_{r \in [ e^{-t-100}   , e^{-t} ]} \sup_{w\in B_z(e^{-t}-r )} |h_r(w) - h_{e^{-t}}(z)|  
\eqD \sup_{r \in [e^{-100} , 1]} \sup_{w\in B_0(1-r)}   |h_r(w) - h_1(0)| .
\eqe 
The process $\{h_r(w) - h_1(0) : r\in [e^{-100},1] , w\in B_0(1-r) \}$ is centered Gaussian and continuous. Furthermore, the variance of $h_r(w) - h_1(0)$ is bounded above by a universal constant for $r\in [e^{-100},1]$ and $ w\in B_{1-r}(0)$. 
By the Borell-TIS inequality~\cite{borell-tis1,borell-tis2} (see, e.g.,~\cite[Theorem 2.1.1]{adler-taylor-fields}), it follows that
\eqbn
\BB E\left[ \sup_{r \in [e^{-100} , 1]} \sup_{w\in B_0(1-r)} |h_r(w) - h_1(0)| \right] < \infty 
\eqen
is a finite, universal constant and there are universal constants $b_0 , b_1 > 0$ such that for each $S > 1$,
\eqb \label{eqn-circle-max-borell0}
\BB P\left[ \sup_{r \in [e^{-100} , 1]} \sup_{w\in B_0(1-r)} |h_r(w) - h_1(0)|  > \log S \right] \leq b_0 e^{-b_1 (\log S)^2} .
\eqe 
Note that here we absorbed the expectation of the supremum (which we know is a universal constant) into the constants $b_0$ and $b_1$. 
By~\eqref{eqn-circle-max-law} and \eqref{eqn-circle-max-borell0}, 
\eqb \label{eqn-circle-max-borell}
\BB P\left[ \sup_{r \in [ e^{-t-100}   , e^{-t} ]} \sup_{w\in B_z(e^{-t}-r )}   |h_r(w) - h_{e^{-t}}(z)|    > \log S \right] \leq b_0 e^{-b_1 (\log S)^2} .
\eqe 

Combining~\eqref{eqn-cond-reg-dist}, \eqref{eqn-cond-reg-dist-lower}, and~\eqref{eqn-circle-max-borell} shows that
\eqbn
\BB P[M_z(e^{-t}) > S] \leq c_0 e^{-c_1 (\log S)^2/(\log\log S)^2}
\eqen
for universal constants $c_0,c_1>0$. This together with the independence argument at the beginning of the proof gives~\eqref{eqn-cond-reg}.
\end{proof}

Using Lemma~\ref{lem-max-and-circle} and a union bound (over certain carefully chosen sets) we can get bounds for $M_z(e^{-t})$ which are uniform over all points in a bounded subset of $  e^{-t-100}\BB Z^2$.

\begin{lem} \label{lem-max-and-circle}
Fix a bounded open set $U\subset\BB C$, a number $\alpha > 0$, and a number $\zeta \in (0,1/2)$. 
Let $M_z(e^{-t})$ be as in Lemma~\ref{lem-cond-reg}. 
There are constants $  c_0,c_1 > 0$ depending only on $U,\alpha,\zeta$ such that for each $t  \geq 1$, it holds with probability at least $1 -  c_0 \exp\left( - c_1 (\log t)^{1+\zeta} \right)$ that the following is true. For each $ z\in (e^{-t-100} \BB Z^2) \cap U$ which satisfies $h_{e^{-t}}(z) - h_1(z) \leq 2 t  - \alpha \log t$, we have
\eqbn
M_z(e^{-t}) \leq \exp\left(   \left[ 2 t -  ( h_{e^{-t}}(z) - h_1(z) ) \right]^{1/2+\zeta} \right) .
\eqen
\end{lem}

When we apply Lemma~\ref{lem-max-and-circle}, we will take $\alpha  \in (0,1/4)$, so that by Proposition~\ref{prop-gff-max} a.s.\ for each large enough $n\in\BB N$, the constraint $h_{e^{-n}}(z) - h_1(z) \leq 2 n  - \alpha \log n$ is satisfied for all $z\in (e^{-n-100} \BB Z^2) \cap U$. 

\begin{proof}[Proof of Lemma~\ref{lem-max-and-circle}]
\noindent\textit{Step 1: partitioning the set of ``bad" points.}
We will break up the points $z\in (e^{-t-100} \BB Z^2) \cap U$ based on the value of $ 2 t - (h_{e^{-t}}(z) - h_1(z)) $. 
Let
\eqb  \label{eqn-max-circle-n}
k_t := \lfloor  \alpha  \log  t   \rfloor   \quad \text{and} \quad K_t := \lceil  2 t \rceil  .
\eqe 
We note that if 
\eqbn
0 \leq h_{e^{-t}}(z) - h_1(z) \leq 2 t - \alpha \log t ,
\eqen
then $ 2 t - (h_{e^{-t}}(z) - h_1(z))$ belongs to $[k,k+1]$ for some $k \in [k_t , K_t - 1]_{\BB Z}$. 

For $k \in [k_t,K_t-1]_{\BB Z}  $, let $Z_t^k$ be the set of $z\in(e^{-t-100} \BB Z^2) \cap U$ such that 
\eqbn
 M_z(e^{-t})  >  \exp\left(  \left[ 2 t -  ( h_{e^{-t}}(z) - h_1(z) ) \right]^{1/2+\zeta}  \right)   \quad \text{and} \quad
  2 t -  (h_{e^{-t}}(z) - h_1(z)) \in [k ,k+1]
\eqen
Also let $Z_t^{K_t}$ be the set of $z\in(e^{-t-100} \BB Z^2) \cap U$ such that 
\eqb \label{eqn-max-circle-neg}
 M_z(e^{-t})  >  \exp\left(  \left[ 2 t -  ( h_{e^{-t}}(z) - h_1(z) ) \right]^{1/2+\zeta}  \right)  \quad \text{and} \quad
     h_{e^{-t}}(z) - h_1(z)  \leq 0 .
\eqe 

We want to show that
\eqb \label{eqn-max-and-circle-show0}
\BB P\left[ \bigcup_{k=k_t}^{K_t} Z_t^k = \emptyset\right] \geq 1 - c_0 \exp\left( - c_1 (\log t)^{1+\zeta} \right)
\eqe 
for constants $c_0,c_1 > 0$ as in the lemma statement. 
We will prove the lemma by showing that
\eqb \label{eqn-max-and-circle-show}
\BB E\left[ \sum_{k=k_t}^{K_t}  \# Z_t^k \right] \leq  c_0 \exp\left( - c_1 (\log t)^{1+\zeta} \right) 
\eqe 
which immediately implies~\eqref{eqn-max-and-circle-show0} via Markov's inequality. 

Let us now prove~\eqref{eqn-max-and-circle-show}. 
Throughout the rest of the proof, we let $c_0$ and $c_1$ denote constants which depend only on $U,\alpha,\zeta$ and which may change from line to line. 
\medskip

\noindent\textit{Step 2: estimates for $M_z(e^{-t})$ and $ h_{e^{-t}}(z) -h_1(z) $.}
By Lemma~\ref{lem-cond-reg}, for each $z \in \BB C$ and each $t  \geq 0$, a.s.\ 
\eqb  \label{eqn-use-cond-reg}
\BB P\left[ M_z(e^{-t}) > S \,|\, h_{e^{-t}}(z) - h_1(z) \right] \leq c_0  e^{-c_1  (\log S)^2/(\log\log S)^2}  , \quad\forall S > 2 .
\eqe 
Hence, for any $k \in \BB N$ and any $x\in [k , k+1]$, 
\allb \label{eqn-max-exp-tail}
\BB P\left[ M_z(e^{-t}) >  \exp\left(  x^{1/2+\zeta} \right) \,|\, 2 t -  (h_{e^{-t}}(z) - h_1(z)) = x\right] 
&\leq c_0  \exp\left( - c_1  \frac{x^{1+ 2\zeta}}{\log x} \right)  \notag\\
&\leq c_0  \exp\left( - c_1  \frac{k^{1+2\zeta}}{\log k} \right)  .
\alle 

On the other hand, the random variable $h_{e^{-t}}(z) - h_1(z)$ is centered Gaussian with variance $t$, so if $k \in [k_t, K_t -1]_{\BB Z}$, then
\allb \label{eqn-circle-tail}
\BB P\left[ 2 t -  (h_{e^{-t}}(z) - h_1(z)) \leq k + 1 \right]
&= \BB P\left[  h_{e^{-t}}(z) - h_1(z) \geq 2 t - (k+1)  \right] \notag\\
&\leq \exp\left( - \frac{(2 t - k-1)^2}{ 2t} \right) \notag\\
&\leq c_0   \exp\left( - 2t +  2 (k+1) \right) ,
\alle 
where in the last line we dropped the term $-(k+1)^2/(2t) $ inside the exponential.
\medskip

\noindent\textit{Step 3: estimates for $\#Z_t^k$.}
Combining~\eqref{eqn-max-exp-tail} with~\eqref{eqn-circle-tail} shows that for $z\in (e^{-t-100} \BB Z^2) \cap U$ and $k \in [k_t,K_t -1]_{\BB Z}$, 
\eqb  \label{eqn-max-circle-pt}
\BB P\left[ z\in Z_t^k \right] 
\leq c_0  \exp\left( - 2t +  2 (k+1)  - c_1 \frac{ k^{1+2\zeta}}{\log k}  \right) 
\leq c_0  \exp\left(  -2t -  c_1   k^{1+ \zeta}   \right)
\eqe 
where in the second inequality we used that $\log k$ is bounded above by a $\zeta$-dependent constant times $k^\zeta$.
To treat the case of $Z_t^{K_t}$, as defined in~\eqref{eqn-max-circle-neg}, we apply~\eqref{eqn-use-cond-reg} with $S = (2t)^{1/2+\zeta}$ and take unconditional expectations of both sides to get
\eqb  \label{eqn-max-circle-pt-neg}
\BB P\left[ z\in Z_t^{K_t} \right] \leq c_0 \exp\left( - c_1 \frac{ t^{1+2\zeta}}{\log t} \right) \leq c_0 \exp\left( - 2t -  c_1   K_t^{1+ \zeta}  \right)
\eqe 
where in the second inequality we used that $K_t = \lfloor 2t \rfloor$, so $K_t^{1+\zeta}$ is much bigger than $2t$. 
Hence~\eqref{eqn-max-circle-pt} also holds for $k=K_t$. 

By~\eqref{eqn-max-circle-pt} and~\eqref{eqn-max-circle-pt-neg} and a union bound over $O_t(e^{2t})$ points in $(e^{-t-100} \BB Z^2)\cap U$, we now obtain
\eqbn
\BB E\left[  \# Z_t^k \right] \leq   c_0   \exp\left( -  c_1   k^{1+ \zeta}   \right)  ,\quad \forall k \in [k_t,K_t ]_{\BB Z}.
\eqen
Hence
\alb
\BB E\left[ \sum_{k=k_t}^{K_t} \# Z_t^k \right]
&\leq c_0 \sum_{k = k_t}^{K_t}     \exp\left( -  c_1   k^{1+ \zeta}   \right)  \\
&\leq c_0   \exp\left( -  c_1   k_t^{1+ \zeta}   \right) \\
&\leq c_0 \exp\left( - c_1 (\log t)^{1+\zeta} \right)  ,
\ale
where in the last line we used the definition of $k_t$ from~\eqref{eqn-max-circle-n}. This gives~\eqref{eqn-max-and-circle-show}, which (as explained above) concludes the proof.
\end{proof}

\begin{proof}[Proof of Lemma~\ref{lem-all-scales}]
The lemma is an easy consequence of Lemma~\ref{lem-max-and-circle}, Proposition~\ref{prop-gff-max}, and the Borel-Cantelli lemma. 
We start with some preliminary estimates.
Since $U$ is bounded and $z\mapsto h_1(z)$ is continuous, a.s.\ 
\eqb \label{eqn-use-h_1-cont}
C_0 := \sup_{z\in U} |h_1(z)|  < \infty .
\eqe 
With a view toward applying Proposition~\ref{prop-gff-max}, let $\alpha < \alpha' <  1/4$. 
By Proposition~\ref{prop-gff-max} with $\alpha'$ in place of $\alpha$, it is a.s.\ the case that for each large enough $n\in\BB N$ (random), we have
\eqbn
|h_{e^{-n}}(z)| \leq 2 n - \alpha' \log n ,  \quad \forall z \in (e^{-n-100} \BB Z^2) \cap U .
\eqen
We have $(\alpha' -\alpha )\log n > C_0$ for each large enough $n\in\BB N$, so a.s.\ for each large enough $n\in\BB N$,  
\eqb \label{eqn-use-gff-max}
|h_{e^{-n}}(z)| \leq 2 n - \alpha  \log n  - C_0  ,  \quad \forall z \in (e^{-n-100} \BB Z^2) \cap U . 
\eqe  
 
For $n\in\BB N$, let $E_n$ be the event of Lemma~\ref{lem-max-and-circle} with $t = n$, i.e., $E_n$ is the event that for each $ z\in (e^{-n-100} \BB Z^2) \cap U$ with $h_{e^{-n}}(z) - h_1(z) \leq 2 n - \alpha \log n$, we have
\eqbn
M_z(e^{-n}) \leq \exp\left(   \left[ 2 n -  (h_{e^{-n}}(z) - h_1(z)) \right]^{1/2+\zeta} \right)  .
\eqen 
By Lemma~\ref{lem-max-and-circle}, there are constants $  c_0,c_1 > 0$ depending only on $U,\alpha,\zeta$ such that for each $n\in\BB N$, 
\eqbn
\BB P[E_n] \geq 1 -  c_0 \exp\left( - c_1 (\log n)^{1+\zeta} \right) .
\eqen
The quantity $c_0 \exp\left( - c_1 (\log n)^{1+\zeta} \right)$ is summable, so the Borel-Cantelli lemma implies that a.s.\ $E_n$ occurs for each large enough $n\in\BB N$.

Now assume that $n$ is large enough so that $E_n$ occurs and~\eqref{eqn-use-gff-max} holds. By~\eqref{eqn-use-h_1-cont} and~\eqref{eqn-use-gff-max}, for each $z\in (e^{-n-100} \BB Z^2) \cap U$,  
\alb
h_{e^{-n}}(z) - h_1(z) 
 \leq  (2 n - \alpha  \log n  - C_0) + C_0 \leq 2n-\alpha \log n .
\ale
Hence the condition in the definition of $E_n$ holds for every $z\in (e^{-n-100}\BB Z^2)\cap U$, so for every such $z$,
\alb
M_z(e^{-n})
&\leq \exp\left(   \left[ 2 n -  (h_{e^{-n}}(z) - h_1(z)) \right]^{1/2+\zeta} \right) \\
&\leq \exp\left(   \left[ 2 n -   h_{e^{-n}}(z)  + C_0 \right]^{1/2+\zeta} \right) \quad \text{(by~\eqref{eqn-use-h_1-cont})} \\
&\leq e^{C_0^{1/2+\zeta}} \exp\left( \left[2n - h_{e^{-n}}(z)\right]^{1/2+\zeta}\right) 
\ale
where in the last line, we used that $x\mapsto x^{1/2+\zeta}$ is concave, hence subadditive. This shows that a.s.~\eqref{eqn-all-scales} holds with $C=e^{C_0^{1/2+\zeta}}$ for each sufficiently large $n\in\BB N$. By increasing $C$, we can arrange that in fact~\eqref{eqn-all-scales} holds for all $n\in\BB N$. 
\end{proof}

\subsection{Stringing together paths at different scales}
\label{sec-dist-to-int}

Fix a bounded open set $U\subset\BB C$ and a parameter $\zeta \in (0,1/2)$.
In this subsection, we will build  on Lemma~\ref{lem-all-scales} to prove the following estimate.

\begin{lem} \label{lem-dist-to-int}
Let $\alpha \in (0,1/4)$. Almost surely, there exists a random $C \in (0,\infty)$ such that the following is true for each $z\in U$. 
\begin{enumerate}
\item We have
\eqb \label{eqn-dist-to-int-circle}
h_{e^{-t}}(z) \leq 2t - \alpha \log t + C ,\quad \forall t \geq 0 .
\eqe
\item For each $n,m\in\BB N$ with $m < n$, 
\allb \label{eqn-dist-to-int}
&D_h\left( \bdy B_z( e^{-n-50} ) , \bdy  B_z(e^{-m} )    \right)  \notag\\
&\qquad\qquad \leq C \sum_{k=m }^{n-1} \min_{t\in [k,k+1]}  \exp\left( \xi h_{e^{-t}}(z) -    \xi Q t   +   (2+2\xi) \left| 2 t - h_{e^{-t}}(z)  \right|^{1/2+\zeta}     \right)  .
\alle 
\item For each $n\in\BB N$, 
\allb \label{eqn-dist-to-int-around}
&D_h\left(\text{around $\BB A_z( e^{-n-1} ,  e^{-n} )$} \right) \notag\\
&\qquad\qquad \leq C \min_{t\in [n+1,n+2]}  \exp\left( \xi h_{e^{-t}}(z) -    \xi Q t   +   (2+2\xi)  \left| 2 t - h_{e^{-t}}(z)  \right|^{1/2+\zeta}     \right) .
\alle
\end{enumerate}
\end{lem} 
 
The estimate~\eqref{eqn-dist-to-int} is of fundamental importance for our argument since it allows us to bound distances between different scales in terms of the circle average process $t\mapsto h_{e^{-t}}(z)$, which is a linear Brownian motion so can be easily estimated. 
The estimate~\eqref{eqn-dist-to-int-around} will be used to ``link up" paths between pairs of concentric circles with possibly different center points. 

To begin the proof of Lemma~\ref{lem-dist-to-int}, let us first record what we get from Lemma~\ref{lem-all-scales}.
By Lemma~\ref{lem-all-scales}, a.s.\ there exists a random $C_0  \in (0,\infty)$ such that~\eqref{eqn-all-scales} holds with $C_0$ in place of $C$ and the Eucldiean 1-neighborhood $B_1(U)$ in place of $U$. 
By the definition of $M_z(e^{-n})$ from Definition~\ref{def-triple-max}, a.s.\ the following is true for each $n\in\BB N$ and each $z\in (e^{-n-100}\BB Z^2)\cap B_1(U)$. 
\begin{enumerate}[A.]
\item \label{item-dist-sum-around} There is a path $P^\circ_n(z)$ in the annulus $A_z^\circ(e^{-n}) = \BB A_z(e^{-n-51/100} , e^{-n-1/2} )$ which disconnects the inner and outer boundaries of this annulus and has $D_h$-length at most 
\eqb \label{eqn-dist-sum-around}
\exp\left( \xi h_{e^{-n}}(z)  - \xi Q n \right) M_z(e^{-n}) 
\leq C_0 \exp\left( \xi h_{e^{-n}}(z) - \xi Q n  +  \left| 2n - h_{e^{-n}}(z)\right|^{1/2+\zeta}\right) .
\eqe 
\item \label{item-dist-sum-across} There is a path $P^\parallel_n(z)$ between the inner and outer boundaries of $A_z^\parallel(e^{-n})$, i.e., between the circles $\bdy B_z(e^{-n-100} )$ and $\bdy B_z( e^{-n} )$, whose $D_h$-length is at most the right side of~\eqref{eqn-dist-sum-around}.  
\item \label{item-dist-sum-circle} We have 
\eqbn
\sup_{r \in [e^{-n-100} , e^{-n} ]} \sup_{w\in B_z( e^{-n} - r  )}  |h_r(w) - h_{e^{-n}}(z)|   
\leq \log M_z(e^{-n}) 
\leq  \left| 2  n - h_{e^{-n}}(z)  \right|^{1/2+\zeta}  + \log C_0 .
\eqen
\end{enumerate}

With $\alpha \in (0,1/4)$ as in Lemma~\ref{lem-dist-to-int}, let $\alpha_0 \in (\alpha,1/4)$. 
By Proposition~\ref{prop-gff-max}, after possibly replacing $C_0$ by a larger random variable we can arrange that a.s.\
\eqb \label{eqn-use-gff-max'}
|h_{e^{-n}}(z)| \leq 2 n - \alpha_0   \log n + \log C_0   ,\quad \forall n\in\BB N, \quad \forall z \in (e^{-n-100} \BB Z^2) \cap U  .
\eqe  

The above estimates only apply to points in $(e^{-n-100}\BB Z^2) \cap U$. The following lemma, which is a straightforward consequence of condition~\ref{item-dist-sum-circle} and~\eqref{eqn-use-gff-max'}, will allow us to extend our estimates to a general choice of $z\in U$.

\begin{lem} \label{lem-shift-bound}
Almost surely, there exists a random $C_1 \in (0,\infty)$ such that the following is true for each $z\in U$. 
For $n\in\BB N$, let $z_n \in (e^{-n-100}\BB Z^2) \cap B_1(U)$ be chosen so that 
\eqb \label{eqn-shift-pt}
z\in B_{z_n}(  e^{-n-99} ) .
\eqe 
For each $t\in [n+1,n+2]$, 
\eqb \label{eqn-shift-bound}
h_{e^{-t}}(z)
\leq h_{e^{-n}}(z_n) +   2  \left| 2t - h_{e^{-t}}(z) \right|^{1/2+\zeta}  + C_1  
\eqe
and 
\eqb\label{eqn-shift-error}
\left| 2n - h_{e^{-n}}(z_n )\right|^{1/2+\zeta} \leq 2  \left| 2t - h_{e^{-t}}(z) \right|^{1/2+\zeta}  +  C_1 .
\eqe
\end{lem} 
\begin{proof} 
Throughout the proof, all statements are required to hold a.s.\ for every possible choice of $z,n,$ and $t$. 
Let $n\in\BB N$ and $t\in [n+1,n+2]$. 
By condition~\ref{item-dist-sum-circle} above applied with $r = e^{-t}$, $z_n$ in place of $z$, and $z$ in place of $w$, we get that for each $t\in [n+1,n+2]$, 
\eqb  \label{eqn-shift-center}
|h_{e^{-t}}(z) - h_{e^{-n}}(z_n)|  \leq  \left| 2  n - h_{e^{-n}}(z_n)  \right|^{1/2+\zeta}  + \log C_0 .
\eqe 
Consequently, we can use the triangle inequality followed by the subadditivity of $x\mapsto x^{1/2+\zeta}$ to get
\allb  \label{eqn-shift-error0}
\left| 2n - h_{e^{-n}}(z_n )\right|^{1/2+\zeta}
&\leq \left|\: |2n - h_{e^{-t}}(z)|   +        \left| 2  n - h_{e^{-n}}(z_n)  \right|^{1/2+\zeta}        + \log C_0         \right|^{1/2+\zeta} \notag\\
&\leq \left| 2n - h_{e^{-t}}(z) \right|^{1/2+\zeta}   +      \left| 2  n - h_{e^{-n}}(z_n)  \right|^{(1/2+\zeta)^2}        + |\log C_0|^{1/2+\zeta} .
\alle
By~\eqref{eqn-use-gff-max'}, there is a random $n_* \in \BB N$ which does not depend on $z$ such that for $n\geq n_*$, we have $2n - h_{e^{-n}}(z_n) \geq \alpha_0 \log n - \log C_0 \geq  2^{ 1/(1/2+\zeta)}$, which implies that
\eqb \label{eqn-shift-error-half}
\left| 2  n - h_{e^{-n}}(z_n)  \right|^{(1/2+\zeta)^2}  \leq \frac12 \left| 2  n - h_{e^{-n}}(z_n)  \right|^{ 1/2+\zeta }  .
\eqe 
Plugging~\eqref{eqn-shift-error-half} into~\eqref{eqn-shift-error0} shows that for $n\geq n_*$, 
\eqbn
\left| 2n - h_{e^{-n}}(z_n )\right|^{1/2+\zeta}
\leq \left| 2n - h_{e^{-t}}(z) \right|^{1/2+\zeta}   +     \frac12 \left| 2  n - h_{e^{-n}}(z_n)  \right|^{1/2 + \zeta}        + |\log C_0|^{1/2+\zeta}
\eqen 
Re-arranging shows that for $n\geq n_*$, 
\alb
\left| 2n - h_{e^{-n}}(z_n )\right|^{1/2+\zeta} 
&\leq 2  \left| 2n - h_{e^{-t}}(z) \right|^{1/2+\zeta}  +  2 |\log C_0|^{1/2+\zeta} \\
&\leq 2  \left| 2t - h_{e^{-t}}(z) \right|^{1/2+\zeta}  +  2 |\log C_0|^{1/2+\zeta}  ,
\ale
since $t\geq n$. 
This gives~\eqref{eqn-shift-error} with $C_1 = 2 |\log C_0|^{1/2+\zeta}$ for $n\geq n_*$. 
Due to the continuity of the circle average process, we can increase $C_1$ to get that~\eqref{eqn-shift-error} also holds for $n \leq n_*$. 
By combining~\eqref{eqn-shift-center} and~\eqref{eqn-shift-error}, we obtain~\eqref{eqn-shift-bound} with $C_1 + \log C_0 $ in place of $C_1$. 
\end{proof}

\begin{figure}[ht!]
\begin{center}
\includegraphics[scale=1]{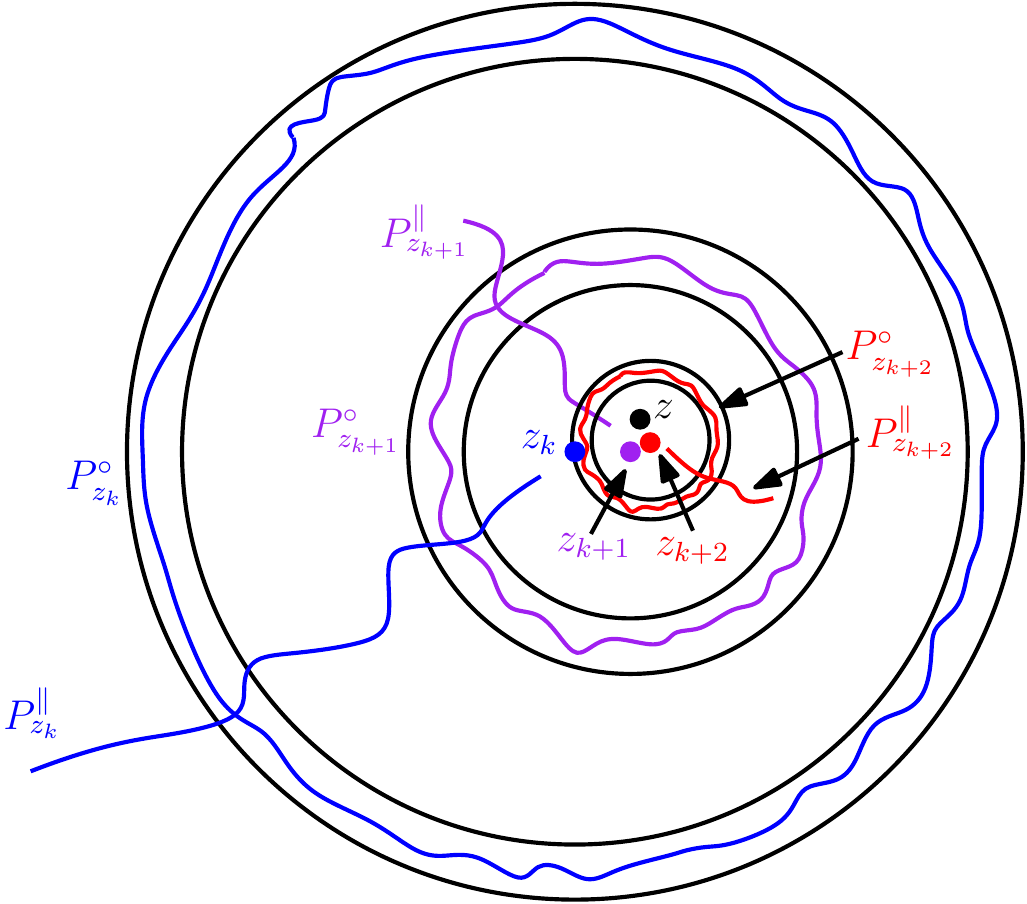} 
\caption{\label{fig-dist-to-int} Illustration of the proof of~\eqref{eqn-dist-to-int} of Lemma~\ref{lem-dist-to-int}. The union of the paths $P_k^\circ(z_k)$ and $P_k^\parallel(z_k)$ from conditions~\ref{item-dist-sum-around} and~\ref{item-dist-sum-across} for three consecutive values of $k$ is connected. To get~\eqref{eqn-dist-to-int}, we sum the $D_h$-lengths of these paths over all $k\in \{m-1,\dots,n-2\}$. 
}
\end{center}
\end{figure}

\begin{proof}[Proof of Lemma~\ref{lem-dist-to-int}]
\noindent\textit{Step 1: proof of~\eqref{eqn-dist-to-int-circle}.}
For $z\in U$ and $n\in\BB N$, let $z_n \in (e^{-n-100}\BB Z^2) \cap U$ be as in Lemma~\ref{lem-shift-bound}.
By~\eqref{eqn-use-gff-max'} and~\eqref{eqn-shift-bound}, if $t\in [n+1,n+2]$ such that $h_{e^{-t}}(z) \geq 2 t - \alpha \log t$, then 
\eqb
h_{e^{-t}}(z)
\leq 2n - \alpha_0 \log n +   2  \left| \alpha \log t \right|^{1/2+\zeta}  + \log C_0 +  C_1 .
\eqe
Since $\alpha < \alpha_0$, this last quantity is smaller than $2 t-\alpha \log t$ if $n$ is sufficiently large (how large is random and does not depend on $z$). Hence, there exists a random $n_*\in\BB N$ such that for each $n\geq n_*$, we have $h_{e^{-t}}(z) \leq 2t -\alpha \log t$ for each $t\in [n+1,n+2]$ and each $z\in U$.
By the continuity of the circle average process, the quantity $\sup_{z\in U} \sup_{t\in [0,n_*+1]} |h_{e^{-t}}(z)|$ is a.s.\ finite.
By choosing $C$ to be 100 times this quantity, say, we get~\eqref{eqn-dist-to-int-circle}. 
\medskip

\noindent\textit{Step 2: bounding lengths of paths.}
By applying~\eqref{eqn-shift-bound} and~\eqref{eqn-shift-error} to estimate the right side of~\eqref{eqn-dist-sum-around}, we get that for each $t\in [n+1,n+2]$, the $D_h$-lengths of the paths $P^\circ_n(z_n)$ and $P^\parallel_n(z_n)$ from conditions~\ref{item-dist-sum-around} and~\ref{item-dist-sum-across} are each bounded above by
\alb
& C_0 \exp\left( \xi h_{e^{-n}}(z) - \xi Q n  +  \left| 2n - h_{e^{-n}}(z)\right|^{1/2+\zeta}\right) \notag\\
&\qquad \leq C_0 \exp\left( \xi h_{e^{-t}}(z) - \xi Q n  + (2+2\xi) \left| 2t -h_{e^{-t}}(z)\right|^{1/2+\zeta} + 2 C_1  \right)  \notag\\
&\qquad\leq  C_0 e^{2C_1 + 2\xi Q}  \exp\left( \xi h_{e^{-t}}(z) - \xi Q t  + (2+2\xi) \left| 2 t -h_{e^{-t}}(z)\right|^{1/2+\zeta} \right) .
\ale
Hence, the $D_h$-lengths of each of these paths is bounded above by
\eqb \label{eqn-length-int}
C  \min_{t\in [n+1,n+2]}  \exp\left( \xi h_{e^{-t}}(z) - \xi Q t  + (2+2\xi) \left| 2 t -h_{e^{-t}}(z)\right|^{1/2+\zeta} \right)  
\eqe 
for $C = C_0 e^{2C_1 + 2\xi Q}$.  
\medskip

\noindent\textit{Step 3: proof of~\eqref{eqn-dist-to-int-around}.}
For $n\in\BB N$, then path $P^\circ_n(z_n)$ disconnects the inner and outer boundaries of the annulus $A_{z_n}^\circ(e^{-n}) = A_{z_n}(e^{-n-51/100} , e^{-n-1/2}) $.
Since $|z-z_n| \leq e^{-n-99}$, this path also disconnects the inner and outer boundaries of the annulus $\BB A_z(e^{-n-1} , e^{-n})$.  
Hence~\eqref{eqn-dist-to-int-around} follows from~\eqref{eqn-length-int} for $P_n^\circ(z_n)$. 
\medskip

\noindent\textit{Step 4: proof of~\eqref{eqn-dist-to-int}.} 
 Let $n,m \in \BB N$ with $m < n$. 
We will prove~\eqref{eqn-dist-to-int} by chaining together the paths $P_k^\circ(z_k)$ and $P_k^\parallel(z_k)$ from conditions~\ref{item-dist-sum-across} and~\ref{item-dist-sum-around} above for $k\in \{m-1,\dots,n-2\}$, where $z_k$ is as in Lemma~\ref{lem-shift-bound}. See Figure~\ref{fig-dist-to-int} for an illustration.
  
For each $k\in \{m-1,\dots,n-2\}$, the path $P_k^\parallel(z_k)$ goes from $\bdy B_{z_k}( e^{-k-100} )$ to $\bdy B_{z_k}( e^{-k} )$.   
We have 
\eqbn
|z_{k+1} - z_k| \leq |z_{k+1} - z| + |z_k - z| \leq 2 e^{-k-99} , 
\eqen
so the annuli $A_{z_k}^\circ( e^{-k}) $ and $A_{z_{k+1}}^\circ(  e^{-k-1} )$ are each contained in $\BB A_{z_k}( e^{-k} , e^{-k-100} )$ and disconnect the inner and outer boundaries of this last annulus. 
Hence $P_k^\parallel(z_k)$ intersects the paths $P_k^\circ(z_k)$ and $P_{k+1}^\circ(z_{k+1})$. 
Therefore, the union of the paths $P_k^\parallel(z_k)$ and $P_k^\circ(z_k)$ for $k\in \{m-1,\dots,n-2\}$ is connected. 
Furthermore, since $|z-z_{m-1}| \leq e^{-m-98}$ and $|z-z_{n-2}| \leq e^{-n-97}$, the path $P_{m-1}^\parallel(z_{m-1})$ intersects $\bdy B_z(e^{-m} )$ and the path $P_{n-2}^\parallel(z_{n-2})$ intersects $\bdy B_z(e^{-n-50} )$. 
Hence summing the estimate~\eqref{eqn-length-int} gives
\alb
&D_h\left(\bdy B_z(e^{-n-50} ) , \bdy  B_z( e^{-m} )    \right) \notag\\
&\qquad\qquad \leq 2 C  \sum_{k=m-1}^{n-2} \min_{t\in [k+1,k+2]}  \exp\left( \xi h_{e^{-t}}(z) - \xi Q t  + (2+2\xi) \left| 2 t -h_{e^{-t}}(z)\right|^{1/2+\zeta} \right) \\
&\qquad \qquad = 2 C  \sum_{k=m }^{n-1} \min_{t\in [k , k+1]}  \exp\left( \xi h_{e^{-t}}(z) - \xi Q t  + (2+2\xi) \left| 2 t -h_{e^{-t}}(z)\right|^{1/2+\zeta} \right) .
\ale 
\end{proof}

\subsection{Specializing to the case when $Q=2$}
\label{sec-dist-to-int'}

We henceforth assume that $\xi = \xi_\crit$, equivalently $Q = 2$. 
In this special case, the estimates of Lemma~\ref{lem-dist-to-int} can be simplified.

\begin{lem} \label{lem-dist-to-int'}
Let $U\subset \BB C$ be a bounded open set and let $\delta\in (0,\xi_\crit)$. Almost surely, there exists a random $C \in (0,\infty)$ such that the following is true for each $z\in U$. 
\begin{enumerate} 
\item For each $n,m\in\BB N$ with $m < n$, 
\eqb \label{eqn-dist-to-int'}
D_h\left( \bdy B_z( e^{-n-50} )  , \bdy B_z( e^{-m} )    \right) 
\leq C \sum_{k=m}^{n-1} \min_{t\in [k , k+1 ]}  \exp\left( (\xi_\crit - \delta) ( h_{e^{-t}}(z) -    2    t )      \right)  .
\eqe 
\item For each $n\in\BB N$, 
\eqb \label{eqn-dist-to-int-around'}
D_h\left(\text{around $\BB A_z( e^{-n-1} , e^{-n} )$} \right) \leq C \min_{t\in [n+1,n+2]}  \exp\left( (\xi_\crit  - \delta) ( h_{e^{-t}}(z) -    2  t )   \right) .
\eqe
\end{enumerate}
\end{lem} 
\begin{proof}
Let $\alpha \in (0,1/4)$ and $\zeta \in (0,1/2)$. 
By~\eqref{eqn-dist-to-int-circle} of Lemma~\ref{lem-dist-to-int}, a.s.\ there is a random $C \in (0,\infty)$ such that for each $z\in U$,
\eqb \label{eqn-dist-to-int-circle-flip}
2t - h_{e^{-t}}(z) \geq   \alpha \log t  -  C ,\quad \forall t \geq 0 .
\eqe 
By~\eqref{eqn-dist-to-int-circle-flip}, for each large enough $t> 0$ it holds for each $z\in U$ that
\eqbn
2t - h_{e^{-t}}(z) \geq 0 \quad \text{and} \quad (2 t - h_{e^{-t}}(z) )^{1/2-\zeta} \geq (2+2\xi_\crit)/\delta .
\eqen
Equivalently, the error term inside the exponential in~\eqref{eqn-dist-to-int} and~\eqref{eqn-dist-to-int-around} in Lemma~\ref{lem-dist-to-int} satisfies
\eqb \label{eqn-dist-to-int-error}
(2+2\xi_\crit)\left| 2 t - h_{e^{-t}}(z)  \right|^{1/2+\zeta}  
\leq \delta (2t - h_{e^{-t}}(z)) .
\eqe
By plugging~\eqref{eqn-dist-to-int-error} into~\eqref{eqn-dist-to-int} and~\eqref{eqn-dist-to-int-around} from Lemma~\ref{lem-dist-to-int}, recalling that $Q=2$, and possibly increasing $C$ to deal with the bounded interval of $t$-values for which~\eqref{eqn-dist-to-int-error} does not hold, we get~\eqref{eqn-dist-to-int'} and~\eqref{eqn-dist-to-int-around'}.
\end{proof}

We have the following simplified form of~\eqref{eqn-dist-to-int-around'} from Lemma~\ref{lem-dist-to-int'}.

\begin{lem} \label{lem-max-around}
Let $U\subset \BB C$ be a bounded open set and let $\theta \in (0,  \xi_\crit / 4)$. 
Almost surely, there is a random $C \in (0,\infty)$ such that for each 
$z\in U$ and each $n\in\BB N$,
\eqb \label{eqn-max-around}
D_h\left(\text{around $\BB A_z( e^{-n-1} , e^{-n} )$} \right) \leq C n^{- \theta } .
\eqe
\end{lem}
\begin{proof}
Let $\delta   \in (0,\xi_\crit)$ and $\alpha \in (0,1/4)$ to be chosen later, depending on $\theta$. 
The estimate~\eqref{eqn-dist-to-int-around'} from Lemma~\ref{lem-dist-to-int'} combined with~\eqref{eqn-dist-to-int-circle} from Lemma~\ref{lem-dist-to-int} shows that a.s.\ there is a random $C \in (0,\infty)$ such that for each $z\in U$ and each $n\in\BB N$,
\eqb
D_h\left(\text{around $\BB A_z(e^{-n-1} , e^{-n} )$} \right) \leq C   \exp\left( -  (\xi_\crit  - \delta) \alpha \log n  \right) = C n^{-(\xi_\crit-\delta)\alpha } .
\eqe
We conclude the proof by choosing $\delta$ sufficiently close to 0 and $\alpha$ sufficiently close to $1/4$ so that $(\xi_\crit-\delta)\alpha \geq \theta$. 
\end{proof}

\subsection{Distances between scales for ``good" and ``bad" points}
\label{sec-good-bad}

We now seek to estimate the sum appearing on the right side of~\eqref{eqn-dist-to-int'}. 
We first consider the easy case, when $h_{e^{-t}}(z)$ does not get too close to $2t$.

\begin{lem} \label{lem-good-pt-dist}
Let $U\subset\BB C$ be a bounded open set, let $\theta  > 0$, and let $\beta  > (1+\theta)/\xi_\crit$.  
Almost surely, there is a random $C \in (0,\infty)$ such that for each $z\in U$ and each $n,m \in \BB N$ with $m < n$ such that
\eqb \label{eqn-good-pt-dist-interval}
h_{e^{-t}}(z)  - h_1(z)  \leq 2 t - \beta \log t   , \quad \forall t \in [ m   , n ]_{\BB Z }  ,
\eqe 
it holds that
\eqb \label{eqn-good-pt-dist}
D_h\left( \bdy B_z(e^{- n - 50} ) , \bdy B_z( e^{-m  } ) \right) 
\leq C m^{-\theta }   .
\eqe 
\end{lem}

Note that, unlike in most of the other lemma statements in this subsection, we allow a general choice of $\theta > 0$ in Lemma~\ref{lem-good-pt-dist} rather than requiring $\theta \in (0,\xi_\crit/4)$. 

\begin{proof}[Proof of Lemma~\ref{lem-good-pt-dist}]
Let $\delta \in (0,\xi_\crit)$ be small. 
By Lemma~\ref{lem-dist-to-int'}, a.s.\ there exists a random $C_0   \in (0,\infty)$ such that for each $z \in U$ satisfying~\eqref{eqn-good-pt-dist-interval},
\allb \label{eqn-good-pt-calc}
D_h\left( B_z(e^{-n-50} ) , \bdy B_z( e^{-m } ) \right) 
&\leq C_0 \sum_{k=m }^{n-1} \min_{t\in [k,k+1 ]}  \exp\left( (\xi_\crit - \delta) ( h_{e^{-t}}(z) -    2    t )      \right) \notag \\
&\leq C_0 \sum_{k=m }^{n-1} \min_{t\in [k,k+1 ]} \exp\left( (\xi_\crit - \delta)(h_1(z) - \beta \log t) \right) \quad \text{(by~\eqref{eqn-good-pt-dist-interval})}   \notag \\
&\leq C_0 \exp\left((\xi_\crit-\delta)\sup_{w\in U} |h_1(w)|  \right) \sum_{k= m }^{n-1}  k^{-(\xi_\crit-\delta)\beta} \quad \notag \\
&\leq C  m^{-(\xi_\crit-\delta)\beta + 1}  
\alle
for some random $C  >0$, where in the last line we use that $\sup_{w\in U} |h_1(w)|$ is a.s.\ finite and $(\xi_\crit - \delta) \beta  >1$. Since $\beta > (1+\theta)/\xi_\crit$, the exponent on the right side of~\eqref{eqn-good-pt-calc} is less than $-\theta$ provided $\delta$ is chosen to be small enough.  
\end{proof}

Lemma~\ref{lem-good-pt-dist} is not sufficient for our purposes since a.s.\ there are points $z \in U$ and arbitrarily small radii $e^{-t}$ for which $h_{e^{-t}}(z)  - h_1(z) \geq 2t - \left(\frac34 + o_t(1)\right) \log t$ (see Proposition~\ref{prop-gff-tight}). We have $3/4 < 1/\xi_\crit$, so we cannot take $\beta \leq 3/4$ in Lemma~\ref{lem-good-pt-dist}. Hence we need an alternative estimate to deal with points $z\in U$ for which $h_{e^{-t}}(z) - h_1(z)$ gets very close to $2t$ for some $t\in [m,n]$. We first consider the case when this happens at the right endpoint, i.e., $h_{e^{-n}}(z) - h_1(z) \geq 2n - \beta \log n$ for an appropriate $\beta>0$. We will eventually reduce to the case when $h_{e^{-n}}(z) - h_1(z)$ is large by considering the \emph{largest} value of $t \in [m,n]$ for which $ h_{e^{-t}}(z) - h_1(z) \geq 2t - \beta \log t$. In fact, for our purposes it is enough to take $m = \lfloor n/2\rfloor$. 

\begin{lem} \label{lem-bad-pt-dist}
Let $U\subset \BB C$ be a bounded open set, let $\theta \in (0 ,   \xi_\crit / 4 )$, and let $\beta > (1+\theta)/\xi_\crit$.  
Almost surely, there is a random constant $C \in (0,\infty)$ such that for each $n\in\BB N$ and each $z\in (e^{-n-100}\BB Z^2) \cap U$ such that
\eqb  \label{eqn-bad-pt-interval0}
h_{e^{-n}}(z)  - h_1(z) \geq 2n - \beta \log n ,
\eqe 
we have
\eqb \label{eqn-bad-pt-dist}
D_h\left( \bdy B_z( e^{-n} ) , \bdy B_z( e^{-n/2} ) \right) 
\leq C n^{-\theta }   .
\eqe 
\end{lem}

Intuitively, the reason why the condition~\eqref{eqn-bad-pt-interval0} helps us get a better estimate is that (by the Gaussian tail bound) the probability that $h_{e^{-n}}(z)  - h_1(z) \geq 2n - \beta \log n$ is small. So, there will not be very many points $z\in (e^{-n-100}\BB Z^2) \cap U$ for which~\eqref{eqn-bad-pt-interval0} holds, and we can take a union bound over all of them. The key inputs in the proof of Lemma~\ref{lem-bad-pt-dist} are Lemma~\ref{lem-dist-to-int'} and the following estimate for a Brownian bridge, which we will apply to the conditional law of $\{h_{e^{-t}}(z) -h_1(z)\}_{t\in [0,n]}$ given $h_{e^{-n}}(z) - h_1(z)$.

\begin{lem} \label{lem-bridge-bound}
Fix $\beta > \alpha > 0$. 
Let $T > 0$, let $x \in [\alpha \log T ,   \beta\log T]$, and let $W$ be a Brownian bridge from 0 to $2T - x$ in time $T$. 
For $S > 0$, 
\eqb \label{eqn-bridge-bound}
\BB P\left[ W_t \leq  2 t - \alpha \log T  \: \forall t \in [T/2 ,T] , \: \int_{T/2}^T \BB 1_{(W_t\geq 2 t - \beta \log T)} \,dt > S \right] 
\leq c_0 \exp\left( - c_1 \frac{S}{(\log T)^2} \right) 
\eqe 
for constants $c_0,c_1>0$ depending only on $\alpha,\beta$. 
\end{lem}

In the application we have in mind, we will take $\alpha \in (0,1/4)$, $\beta > (1+\theta)/\xi_\crit$ as in Lemma~\ref{lem-good-pt-dist}, and $T=n$.
Our Brownian bridge $W_t =  h_{e^{-t}}(z) -h_1(z) $ (conditioned on $h_{e^{-n}}(z) -h_1(z)$) will be forced to stay below $ 2 t - \alpha \log n$ for $t\in [n/2,n]$ by Proposition~\ref{prop-gff-max} (in the form of~\eqref{eqn-dist-to-int-circle} of Lemma~\ref{lem-dist-to-int}). Hence Lemma~\ref{lem-bridge-bound} will provide an upper bound on the Lebesgue measure of the set of times $t\in [n/2,n]$ for which $h_{e^{-t}}(z) -h_1(z) \geq 2t - \beta \log n$. This will allow us to upper-bound the sum on the right side of~\eqref{eqn-dist-to-int'}.

\begin{figure}[ht!]
\begin{center}
\includegraphics[scale=1]{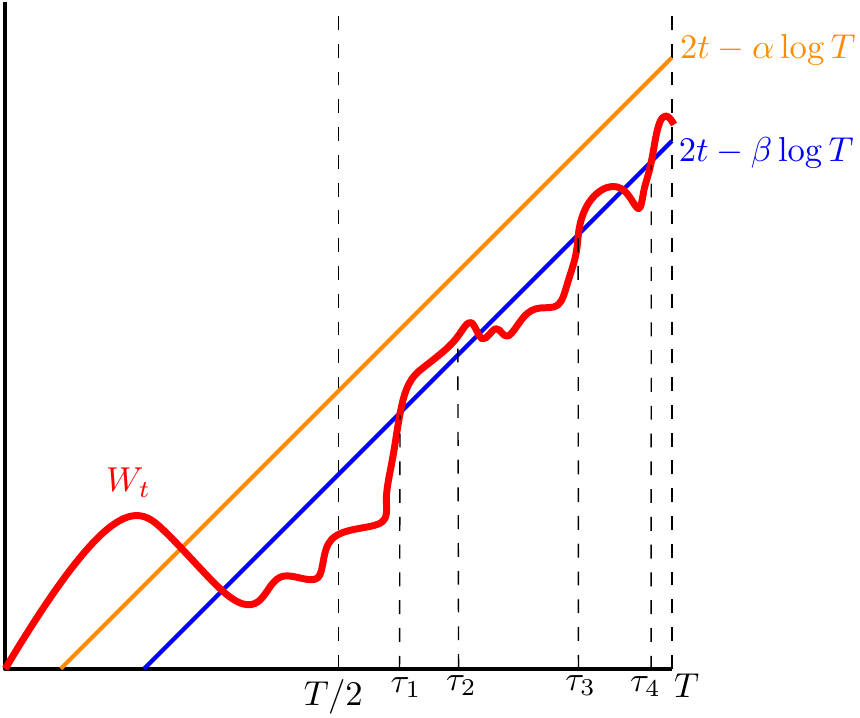} 
\caption{\label{fig-bridge-bound} Illustration of the proof of Lemma~\ref{lem-bridge-bound}. The time $\tau_1$ is the first time after $T/2$ for which $W_t \geq 2t-\beta\log T$ and the time $\tau_k$ is the first time after $\tau_{k-1} + (\log T)^2$ for which $W_t \geq 2t-\beta\log T$, or $\tau_k = T$ if no such time exists. Note that we may have $W_{\tau_k} > 2\tau_k - \beta \log T$, as is the case for $\tau_2$ in the figure. The idea of the proof is that in each interval $[\tau_k , \tau_k + (\log T)^2]$ there is a chance that $W_t$ goes above $2t - \alpha \log T$. From this and an independence argument, we get it is very unlikely that many of the times $\tau_k$ are strictly less than $T$ and $W_t \leq 2t - \alpha \log T$ for all $t\in [T/2,T]$. 
}
\end{center}
\end{figure}

\begin{proof}[Proof of Lemma~\ref{lem-bridge-bound}]
The idea of the proof is that if $t$ is such that $W_t \geq 2 t - \beta \log T$, then by a basic Gaussian estimate there is a positive chance that $W_{t+(\log T)^2}$ will be larger than $2 t - \alpha \log T$. So, it is very unlikely that $W$ spends a lot of time above $2 t-\beta \log T$ while simultaneously staying below $2 t- \alpha \log T$. See Figure~\ref{fig-bridge-bound} for an illustration. 
\medskip

\noindent\textit{Step 1: setup.}
To simplify the calculations, we will first reduce to a Brownian bridge from 0 to 0. 
We can write
\eqbn
W_t = \rng W_t + \frac{t}{T}(2 T - x) = \rng W_t + 2 t -  \frac{t}{T} x
\eqen 
where $\rng W$ is a Brownian bridge from 0 to 0 in time $T$. 
The event whose probability we seek to bound can be written in terms of $\rng W$ as
\eqb \label{eqn-bridge-event}
G = \left\{ \rng W_t \leq    \frac{t}{T} x  - \alpha \log T    \: \forall t \in [T/2 ,T] , \: \int_{T/2}^T \BB 1\left\{\rng W_t \geq  \frac{t}{T} x - \beta \log T   \right\}   \,dt > S \right\} .
\eqe 

Let
\eqbn
\tau_1 := T\wedge \inf\left\{t\geq T/2 :   \rng W_t \geq  \frac{t}{T} x - \beta \log T     \right\} .
\eqen
For $k\geq 2$, inductively let
\eqb \label{eqn-bridge-time-def}
\tau_k := T \wedge \inf\left\{t \geq \tau_{k-1} + (\log T)^2 :  \rng W_t  \geq \frac{t}{T} x - \beta \log T  \right\} . 
\eqe 
\medskip

\noindent\textit{Step 2: reducing to a bound for an event involving the $\tau_k$s.}
We will now upper-bound the probability of the event $G$ of~\eqref{eqn-bridge-event} in terms of the $\tau_k$s.
To this end, let
\eqb \label{eqn-bridge-above}
E_k := \left\{ \tau_k \leq T - 2(\log T)^2,\:  \rng W_{\tau_k + (\log T)^2} >  \frac{\tau_k + (\log T)^2}{T} x  - \alpha \log T   \right\} .
\eqe 
Note that $E_k \in \sigma(\rng W|_{[0,\tau_{k+1}]})$. Furthermore, if $E_k$ occurs for any $k\in\BB N$, then the first condition in the event $G$ of~\eqref{eqn-bridge-event} does not occur. Hence
\eqb \label{eqn-bridge-bigcap}
G\subset \bigcap_{k=1}^\infty E_k^c .
\eqe 
 
Let
\eqb \label{eqn-last-hit}
K := \max\left\{k\in\BB N : \tau_k \leq T - 2(\log T)^2 \right\} .
\eqe 
If $t \in [T/2,T]$ does not belong $\bigcup_{k=1}^K [\tau_k , \tau_k + (\log T)^2] \cup [T-2(\log T)^2,T]$, then $\rng W_t \leq  \frac{t}{T} x  - \beta \log T $. Consequently,
\eqbn
\int_{T/2}^T \BB 1\left\{\rng W_t \geq  \frac{t}{T} x - \beta \log T \right\}   \,dt  \leq (K+2) (\log T)^2 .
\eqen
By~\eqref{eqn-bridge-bigcap} and the definition~\eqref{eqn-bridge-event} of $G$, we therefore have
\eqb  \label{eqn-bridge-event-decomp} 
G\subset \bigcap_{k=1}^\infty E_k^c  \cap \left\{K  > \frac{S}{(\log T)^2} - 2 \right\}  .
\eqe 
\medskip

\noindent\textit{Step 3: proof conditional on a Brownian bridge calculation.}
Let us now bound the probability of the event on the right side of~\eqref{eqn-bridge-event-decomp}. 
Just below, we will show via an elementary Brownian bridge calculation that there is a constant $p = p(\alpha,\beta) \in (0,1)$ such that for each $k\in \{0,1,2,\dots\}$, it holds a.s.\ on the event $\{\tau_k \leq T - 2(\log T)^2\}$ that
\eqb \label{eqn-bridge-prob}
\BB P\left[ E_k \,|\, \rng W|_{[0,\tau_{k }]} \right] \geq p . 
\eqe 
Since $\{\tau_k \leq T - 2(\log T)^2\} \subset E_k$ and $  E_j \in \sigma\left(\rng W|_{[0,\tau_{k }]}\right)$ for each $j \in \{1,\dots,k-1\}$, we deduce from~\eqref{eqn-bridge-prob} that
\allb \label{eqn-bridge-prob'}
&\BB P\left[E_k^c  \cap \{\tau_k \leq T - 2(\log T)^2\}     \,|\, \bigcap_{j=1}^{k-1} E_j^c \cap \{\tau_{k-1} \leq T - 2(\log T)^2\}       \right]\notag\\
&\qquad\qquad \leq \BB P\left[ E_k^c   \,|\, \bigcap_{j=1}^{k-1} E_j^c    \cap \{\tau_k \leq T - 2(\log T)^2\}       \right] \notag\\
&\qquad\qquad  \leq 1 - p .
\alle
Iterating the estimate~\eqref{eqn-bridge-prob'} $k$ times gives
\eqb \label{eqn-bridge-iterate}
\BB P\left[\tau_k \leq T - 2(\log T)^2 , \: \bigcap_{j=1}^k E_j^c \right] \leq (1-p)^k .
\eqe  
We now take $k = \lceil \frac{S}{(\log T)^2} - 2 \rceil$ in~\eqref{eqn-bridge-iterate} and recall the definition~\eqref{eqn-last-hit} of $K$ to get 
\eqbn
\BB P\left[  \bigcap_{k=1}^\infty E_k^c  \cap \left\{K  > \frac{S}{(\log T)^2} - 2 \right\} \right] \leq c_0 \exp\left( - c_1 \frac{S}{(\log T)^2} \right) 
\eqen
for constants $c_0,c_1 >0$ depending only on $\alpha,\beta$. Combining this with~\eqref{eqn-bridge-event-decomp} concludes the proof. 
\medskip

\noindent\textit{Step 4: Brownian bridge calculation.}
It remains to prove~\eqref{eqn-bridge-prob}. The proof is an elementary application of the formulas for the mean and variance of a Brownian bridge with given endpoints, together with some straightforward estimates.

If we condition on $\rng W|_{[0,\tau_k]}$, then the conditional law of $\rng W_{\cdot + \tau_k}$ is that of a Brownian bridge from $\rng W_{\tau_k}$ to 0 in time $T-\tau_k$. In particular, if $\tau_k < T-  (\log T)^2$, then the conditional law of $\rng W_{\tau_k + (\log T)^2}$ is Gaussian with mean 
\eqb \label{eqn-bridge-mean}
\left(1 - \frac{(\log T)^2}{T-\tau_k} \right) \rng W_{\tau_k}    
\eqe 
and variance
\eqb  \label{eqn-bridge-var}
(\log T)^2 - \frac{(\log T)^4}{T-\tau_k} . 
\eqe 

We will now estimate the above formulas for the conditional mean and variance. 
Recall that $T/2 \leq \tau_k \leq T$ and $\rng W_{\tau_k} \geq \frac{\tau_k}{T} x  - \beta \log T $ by~\eqref{eqn-bridge-time-def} and $x  \in [\alpha \log T ,\beta \log T]$ by definition. 
In particular, 
\eqbn
\frac{\tau_k}{T} x - \beta \log T \leq x - \beta \log T \leq 0 .
\eqen
Hence 
\allb \label{eqn-bridge-cond-mean}
\BB E\left[ \rng W_{\tau_k + (\log T)^2} \,|\, \rng W|_{[0,\tau_k]}  \right] 
&= \left(1 - \frac{(\log T)^2}{T-\tau_k} \right) \rng W_{\tau_k}  \notag\\
&\geq \left(1 - \frac{(\log T)^2}{T-\tau_k} \right) \left( \frac{\tau_k}{T} x - \beta \log T \right) \notag\\
&\geq   \frac{\tau_k}{T} x - \beta \log T    \quad \text{(since $\frac{\tau_k}{T} x - \beta \log T \leq 0$)} \notag\\
&\geq  \frac{1}{2} x - \beta \log T  \quad \text{(since $\tau_k \geq T/2$)}  \notag\\
&\geq (\alpha/2 - \beta) \log T \quad \text{(since $x \geq \alpha \log T$)}  .
\alle
Furthermore, using~\eqref{eqn-bridge-var}, we see that if $\tau_k \leq T - 2(\log T)^2$, then 
\eqb \label{eqn-bridge-cond-var}
\op{Var}\left[ \rng W_{\tau_k + (\log T)^2} \,|\, \rng W|_{[0,\tau_k]}  \right] 
\geq \frac12 (\log T)^2 .
\eqe

On the event $\{\tau_k \leq T - 2(\log T)^2\}$, the quantity appearing in the definition~\eqref{eqn-bridge-above} of $E_k$ satisfies
\eqb \label{eqn-bridge-threshold}
 \frac{\tau_k + (\log T)^2}{T} x  - \alpha \log T 
 \leq x - \alpha \log T \leq (\beta-\alpha) \log T .
\eqe 
Hence, on this event
\alb
\BB P\left[ E_k \,|\, \rng W|_{[0,\tau_{k }]} \right]
&\geq  \BB P\left[ \rng W_{\tau_k + (\log T)^2} >  (\beta-\alpha) \log T  \,|\, \rng W|_{[0,\tau_{k }]} \right] \quad \text{(by~\eqref{eqn-bridge-threshold})} \\ 
&\geq p \quad \text{(by~\eqref{eqn-bridge-cond-mean} and~\eqref{eqn-bridge-cond-var})} ,
\ale
where $p \in (0,1)$ depends only on $\alpha,\beta$. 
\end{proof}

\begin{proof}[Proof of Lemma~\ref{lem-bad-pt-dist}]
Let us first explain the main idea of the proof. For a small $\delta\in (0,\xi_\crit)$, Lemma~\ref{lem-dist-to-int'} gives the bound
\eqb \label{eqn-bad-pt-dist-outline}
D_h\left( \bdy B_z( e^{-n} ) , \bdy B_z( e^{-n/2} ) \right) \leq C \int_{\lfloor n/2 \rfloor}^n \exp\left((\xi_\crit-\delta) ( h_{e^{-t}}(z) -2t) \right) \,dt 
\eqe
for a random $C \in (0,\infty)$ (which does not depend on $z$ and $n$). 
To estimate the integral on the right side of~\eqref{eqn-bad-pt-dist-outline}, we consider separately the integrals over the ``good" set of $t\in [\lfloor n/2 \rfloor , n]$ where $h_{e^{-t}}(z) - 2t  \leq -\beta\log n$ and the ``bad" set where $h_{e^{-t}}(z) - 2t  > -\beta\log n$.  
The integral over the ``good" set is bounded above by a constant times $n^{-\theta}$ by our choice of $\beta$. 
Proposition~\ref{prop-gff-max} together with Lemma~\ref{lem-bridge-bound} will allow us to show that the Lebesgue measure of the ``bad" set is small. Furthermore, Proposition~\ref{prop-gff-max} (in the form of Lemma~\ref{lem-dist-to-int}) shows that if $\alpha\in (0,1/4)$, then when $n$ is sufficiently large, we have $h_{e^{-t}}(z) - 2t \leq -\alpha \log n$ for each $t\in [\lfloor n/2\rfloor , n]$. These two facts will allow us to upper-bound the integral over the ``bad" set by a constant times $n^{-\theta}$ as well. 
Let us now proceed with the details.
\medskip

\noindent\textit{Step 1: setup.}
Let $\alpha \in (0,1/4)$ (we will eventually take $\alpha$ to be very close to $1/4$, depending on $\theta$). By~\eqref{eqn-dist-to-int-circle} of Lemma~\ref{lem-dist-to-int} applied with $\alpha' \in (\alpha,1/4)$ in place of $\alpha$, a.s.\ there exists a random $t_* \in \BB N$ such that for each $z\in U$, 
\eqb \label{eqn-use-circle}
h_{e^{-t}}(z)  - h_1(z) \leq 2t - \alpha \log(2t)  ,\quad \forall t \geq t_* .
\eqe

For $n\in\BB N$ and $z\in  U$, let
\allb \label{eqn-bad-dist-event}
G_n(z) 
&:= \left\{ h_{e^{-t}}(z) - h_1(z) \leq  2 t - \alpha \log n  \: \forall t \in [n/2 ,n] \right\} \notag\\
&\qquad \cap\left\{ \int_{n/2}^n \BB 1\left\{  h_{e^{-t}}(z) - h_1(z)   \geq 2 t -  \beta \log n \right\} \,dt > (\log n)^4   \right\} .
\alle
Note that by~\eqref{eqn-use-circle}, a.s.\ for each large enough $n\in\BB N$ the first event in the definition~\eqref{eqn-bad-dist-event} occurs for every $z\in U$. 
The event $G_n(z)$ should be compared to the event of Lemma~\ref{lem-bridge-bound}. 
\medskip

\noindent\textit{Step 2: bound for the probability of $G_n(z) \cap \{h_{e^{-n}}(z) - h_1(z) \in [2n - \beta \log n , 2n - \alpha \log n]\}$.}
By the calculations in~\cite[Section 3.1]{shef-kpz}, for each fixed $z\in U$, the process $t\mapsto h_{e^{-t}}(z) - h_1(z)$ is a standard linear Brownian motion. 
In particular, if $n\in\BB N$, $z\in   U$, and $x \in [\alpha \log n , \beta\log n]$, then the conditional distribution of $t\mapsto h_{e^{-t}}(z) - h_1(z)$ given $\{h_{e^{-n}}(z) - h_1(z) = 2n -  x\}$ is that of a Brownian bridge from 0 to $2n - x$ in time $n$. 
By Lemma~\ref{lem-bridge-bound} applied with $T=n$ and $S = (\log n)^4$, we therefore obtain
\eqb \label{eqn-use-bridge-bound}
\BB P\left[ G_n(z) \,|\, h_{e^{-n}}(z) - h_1(z) \in [2n - \beta \log n , 2n - \alpha \log n] \right] \leq c_0 \exp\left( - c_1 (\log n)^2 \right) 
\eqe
for constants $c_0,c_1>0$ depending only on $\alpha,\beta$. Since $h_{e^{-n}}(z) - h_1(z)$ is centered Gaussian with variance $n$, we also have
\eqb \label{eqn-bad-pt-gaussian}
\BB P\left[ h_{e^{-n}}(z) - h_1(z)  \geq 2n - \beta \log n  \right] 
\leq \exp\left( - \frac{(2n-\beta\log n)^2}{2n} \right)
\leq \exp\left( -2n \right) n^{ 2\beta}  . 
\eqe 
By~\eqref{eqn-use-bridge-bound} and~\eqref{eqn-bad-pt-gaussian}, for each $z\in  U$,
\eqb \label{eqn-bad-pt-combined}
\BB P\left[ G_n(z) ,\,  h_{e^{-n}}(z) - h_1(z) \in [2n - \beta \log n , 2n - \alpha \log n] \right] \leq c_0 \exp\left( -2n  - c_1 (\log n)^2 \right) n^{2\beta }. 
\eqe
\medskip

\noindent\textit{Step 3: a.s.\ bounds for points $z$ such that $h_{e^{-n}}(z) - h_1(z) \geq 2n-\beta\log n$.}
By~\eqref{eqn-bad-pt-combined} and a union bound over $O_n(e^{-2n})$ choices of $z \in (e^{-n-100}\BB Z^2)\cap U$, it holds with probability at least  $1 - O_n\left(  e^{-c_1(\log n)^2} n^{2\beta}       \right) $ that for each such $z$, either $G_n(z)$ does not occur or $h_{e^{-n}}(z) - h_1(z) \notin [2n - \beta \log n , 2n - \alpha \log n]$. 
Recalling the definition~\eqref{eqn-bad-dist-event} of $G_n(z)$, we see that this implies that with probability at least $1 - O_n\left(  e^{-c_1(\log n)^2} n^{2\beta}       \right) $, at least one of the following three conditions holds for each $z\in (e^{-n-100}\BB Z^2)\cap U$:
\begin{enumerate}[$(i)$]
\item $h_{e^{-n}}(z) -h_1(z) < 2n-\beta \log n$; \label{item-trichotomy-upper}
\item If we let 
\eqb \label{eqn-trichotomy-int}
H_n(z,t) := \left\{  h_{e^{-t}}(z) - h_1(z)   \geq 2 t - \beta \log n \right\} ,
\eqe
then $\int_{n/2}^n \BB 1_{H_n(z,t)} \,dt \leq(\log n)^4$; or \label{item-trichotomy-int}
\item $h_{e^{-t}}(z) - h_1(z) >  2 t - \alpha \log n $ for some $t\in [n/2,n]$. \label{item-trichotomy-lower}
\end{enumerate}
By the Borel-Cantelli lemma, a.s.\ this trichotomy holds for all large enough $n\in\BB N$. 

By~\eqref{eqn-use-circle}, a.s.\ for each large enough $n\in\BB N$ the condition~\eqref{item-trichotomy-lower} is not satisfied for any $z\in (e^{-n-100}\BB Z^2)\cap U$. 
Therefore, it is a.s.\ the case that for each large enough $n\in\BB N$, it holds for each $z\in (e^{-n-100}\BB Z^2)\cap U$  such that $h_{e^{-n}}(z)  - h_1(z) \geq 2n - \beta \log n$ that
\eqb \label{eqn-bad-pt-dichotomy}
h_{e^{-t}}(z) - h_1(z) \leq 2t - \alpha \log n , \: \forall t \in [n/2,n] \quad \text{and} \quad 
\int_{n/2}^n \BB 1_{H_n(z,t)} \,dt \leq(\log n)^4 .
\eqe
\medskip

\noindent\textit{Step 4: splitting up the integral.}
Now let $\delta  \in (0,\xi_\crit)$ be small. 
By Lemma~\ref{lem-dist-to-int'}, a.s.\ there exists $C \in (0,\infty)$ such that for each $n\in\BB N$ and each $z\in U$, 
\allb \label{eqn-dist-to-int-half}
D_h\left( \bdy B_z( e^{-n-50} ) , \bdy B_z( e^{-n/2} ) \right) 
&\leq C \sum_{k=\lfloor n /2 \rfloor }^{n-1} \min_{t\in [k , k+1]}  \exp\left( (\xi_\crit-\delta) ( h_{e^{-t}}(z) -    2    t ) \right) \notag \\
&\leq C \int_{\lfloor n/2 \rfloor}^{n } \exp\left( (\xi_\crit-\delta) ( h_{e^{-t}}(z) -    2    t )    \right) \,dt .
\alle 
We now break up the integral on the right side of~\eqref{eqn-dist-to-int-half} based on whether or not the event $H_n(z,t)$ from~\eqref{eqn-trichotomy-int} occurs. 
Using~\eqref{eqn-bad-pt-dichotomy} and~\eqref{eqn-dist-to-int-half}, we get that a.s.\ for each large enough $n\in\BB N$ and each $z\in (e^{-n-100}\BB Z^2)\cap U$  such that $h_{e^{-n}}(z)  - h_1(z) \geq 2n - \beta \log n$, 
\allb \label{eqn-bad-pt-compute}
&D_h\left( B_z(e^{-n-50} ) , \bdy B_z( e^{-n/2} ) \right) \notag\\
&\qquad \leq C \int_{\lfloor n/2 \rfloor}^{n } \exp\left( (\xi_\crit-\delta) ( h_{e^{-t}}(z) -    2     t )     \right) \BB 1_{H_n(z,t)} \,dt \notag \\
&\qquad\qquad  + C \int_{\lfloor n/2 \rfloor}^{n } \exp\left( (\xi_\crit-\delta) ( h_{e^{-t}}(z) -    2    t )    \right) \BB 1_{H_n(z,t)^c}  \,dt \notag \\
&\qquad\leq C \int_{\lfloor n/2 \rfloor}^{n } \exp\left(   (\xi_\crit-\delta) (h_1(z)  - \alpha \log n )     \right) \BB 1_{H_n(z,t)} \,dt \notag \\
&\qquad\qquad  + C \int_{\lfloor n/2 \rfloor}^{n  } \exp\left(     (\xi_\crit  - \delta) (h_1(z)  - \beta \log n )     \right) \BB 1_{H_n(z,t)^c}  \,dt \quad \text{(by~\eqref{eqn-trichotomy-int} and~\eqref{eqn-bad-pt-dichotomy})} \notag \\
&\qquad\leq C \exp\left((\xi_\crit-\delta) \sup_{z\in U} h_1(z) \right) \left[  n^{-\alpha (\xi_\crit - \delta)} \int_{\lfloor n/2 \rfloor}^{n+2}  \BB 1_{H_n(z,t)} \,dt 
 +   n^{-\beta (\xi_\crit-\delta) + 1} \right] \notag \\
&\qquad\leq C \exp\left((\xi_\crit-\delta) \sup_{z\in U} h_1(z) \right) \left[ n^{-\alpha (\xi_\crit - \delta)} (\log n)^4 +   n^{-\beta (\xi_\crit-\delta) + 1}   \right]  \quad \text{(by~\eqref{eqn-bad-pt-dichotomy})}.
\alle

Since $\beta > (1+\theta)/\xi_\crit$, we can choose $\delta >0$ small enough so that $-\beta(\xi_\crit-\delta) + 1 < - \theta$. 
Furthermore, since $\theta <   \xi_\crit / 4$ we can choose $\alpha \in (0,1/4)$ sufficiently close to $1/4$ and $\delta > 0$ small enough so that $-\alpha(\xi_\crit-\delta) < - \theta$. 
The quantity $\sup_{z\in U} h_1(z)$ is a.s.\ finite. Hence the right side of~\eqref{eqn-bad-pt-compute} is bounded above by a positive, finite random variable $C'$ (which does not depend on $z$ or $n$) times $ n^{-\theta}$. 
The bound~\eqref{eqn-bad-pt-compute} holds a.s.\ for all large enough $n\in\BB N$ and all $z\in  (e^{-n-100}\BB Z^2)\cap U$. By possibly increasing $C'$ to deal with finitely many small values of $n$, we obtain~\eqref{eqn-bad-pt-dist}.
\end{proof}

Using Lemma~\ref{lem-shift-bound}, we can extend Lemma~\ref{lem-bad-pt-dist} to an estimate which holds for all $z\in U$, instead of just $z\in (e^{-n-100}\BB Z^2)\cap U$.

\begin{lem} \label{lem-bad-pt-dist'}
Let $U\subset \BB C$ be a bounded open set, let $\theta \in (0 ,  \xi_\crit / 4)$, and let $\beta > (1+\theta)/\xi_\crit$.  
Almost surely, there is a random $C \in (0,\infty)$ such that for each $n\in\BB N$ and each $z\in U$ such that 
\eqb \label{eqn-bad-pt-dist-interval}
\exists t\in [n+1,n+2] \: \text{with} \: h_{e^{-t}}(z)  - h_1(z)  \geq 2 t - \beta \log t ,
\eqe 
it holds that
\eqb \label{eqn-bad-pt-dist'}
D_h\left( \bdy B_z(e^{- n - 40} ) , \bdy B_z( e^{-n/2 -1 } ) \right) 
\leq C n^{-\theta}   .
\eqe 
\end{lem}
\begin{proof}
Let $\beta' > \beta$.
By Lemma~\ref{lem-bad-pt-dist}, applied with $\beta'$ instead of $\beta$, a.s.\ there exists $C \in (0,\infty)$ such that for each $n\in\BB N$ and each $z\in (e^{-n-100}\BB Z^2)\cap U$ such that 
\eqbn
h_{e^{-n}}(z)  - h_1(z) \geq 2n - \beta' \log n ,
\eqen
we have
\eqb \label{eqn-use-bad-pt-dist}
D_h\left( \bdy B_z( e^{- n-50} ) , \bdy B_z(e^{-n/2} ) \right) 
\leq C  n^{-\theta}   .
\eqe 
Furthermore, by Lemma~\ref{lem-dist-to-int}, a.s.\ there exists $t_*  >0$ such that
\eqb \label{eqn-bad-pt-max}
 h_{e^{-t}}(z)  \leq 2t ,\quad \forall z \in U ,\quad \forall t \geq t_* . 
\eqe 

Now let $z\in U$ and assume that~\eqref{eqn-bad-pt-dist-interval} holds.
Let $t\in [n+1,n+2]$ be such that $h_{e^{-t}}(z)  - h_1(z)   \geq 2 t - \beta \log t$. 

Let $z_n \in (e^{-n-100}\BB Z^2)\cap U$ be as in Lemma~\ref{lem-shift-bound}, so that $|z-z_n| \leq e^{-n-99}$. 
By Lemma~\ref{lem-shift-bound}, applied with $\zeta = 1/4$, say, 
\allb \label{eqn-last-hit-shift0}
h_{e^{-n}}(z_n) - h_1(z_n) 
 \geq h_{e^{-t}}(z)  -   2  \left| 2 t - h_{e^{-t}}(z) \right|^{3/4} - |h_1(z_n) |  -  C_1  .
\alle
By~\eqref{eqn-bad-pt-max} and our choice of $t$, we have $h_{e^{-t}}(z) \geq 2t - \beta\log t - |h_1(z)|$ and (if $t\geq t_*$) then
\eqbn
|2t - h_{e^{-t}}(z)| = 2t - h_{e^{-t}}(z) \leq \beta\log t + |h_1(z)| .
\eqen
Applying these last two estimates to the right side of~\eqref{eqn-last-hit-shift0}, then noting that $t\in[n+1,n+2]$, gives
\allb \label{eqn-last-hit-shift}
h_{e^{-n}}(z_n) - h_1(z_n) 
&\geq 2 t - \beta \log t  -  2  \left| \beta \log t\right|^{3/4} - |h_1(z)|   - |h_1(z)|^{3/4} - |h_1(z_n) | -  C_1  \notag \\
&\geq 2 n - \beta  \log n - 2  \left| \beta \log n \right|^{3/4}  - |h_1(z)| - |h_1(z)|^{3/4} - |h_1(z_n) | - C_1' ,
\alle
where $C_1'$ is equal to $C_1$ plus a deterministic positive constant depending only on $\beta$. 

Since $\sup_{w\in U} |h_1(w)|$ is a.s.\ finite and $\beta' > \beta$, a.s.\ there is a random $n_* \geq t_*$, which does not depend on $z$, such that if $n\geq n_*$ then the right side of~\eqref{eqn-last-hit-shift} is at least $2 n  - \beta' \log n$. By~\eqref{eqn-use-bad-pt-dist}, it is a.s.\ the case that if $n\geq n_*$, then
\eqb \label{eqn-bad-pt-dist0'}
D_h\left( \bdy B_{z_n}(  e^{-n-50} ) , \bdy B_{z_n}( e^{-n/2} ) \right) 
\leq C n^{-\theta }   .
\eqe
Since $|z-z_n| \leq e^{-n-99}$, we have
\eqbn
\BB A_z( e^{-n-40} , e^{-n/2-1} ) \subset \BB A_{z_n}( e^{-n-50},e^{-n/2} ) .
\eqen
Therefore~\eqref{eqn-bad-pt-dist0'} implies~\eqref{eqn-bad-pt-dist'} for $n\geq n_*$. Since the $D_h$-distance between any two circles is a.s.\ finite, we see that~\eqref{eqn-bad-pt-dist'} for $n\geq n_*$ implies~\eqref{eqn-bad-pt-dist'} for all $n\in\BB N$ with a possibly larger value of $C$. 
\end{proof}

\subsection{Proof of Theorem~\ref{thm-holder}}
\label{sec-holder}

We now combine Lemmas~\ref{lem-good-pt-dist} and~\ref{lem-bad-pt-dist'} to get an estimate for $D_h\left( \bdy B_z(e^{- n-50 } ) , \bdy B_z( e^{-n/2-1} ) \right) $ which holds uniformly for all $n\in\BB N$ and all $z\in U$.

\begin{lem} \label{lem-all-pt-dist} 
Let $U\subset \BB C$ be a bounded open set and let $\theta \in (0 ,  \xi_\crit / 4)$.
Almost surely, there is a random $C \in (0,\infty)$ such that for each $n\in\BB N$ and each $z\in U$, 
\eqb \label{eqn-all-pt-dist}
D_h\left( \bdy B_z(e^{- n-50 } ) , \bdy B_z( e^{-n/2-1} ) \right) 
\leq C n^{-\theta}   
\eqe 
\end{lem}


\begin{proof}
Let $\beta > (1+\theta)/\xi_\crit$ and consider $n\in\BB N$ and a point $z\in U$. 
We already know from Lemma~\ref{lem-good-pt-dist} (applied with $m = \lfloor n/2\rfloor$) that a.s.\ there exists a random $C \in (0,\infty)$ such that \eqref{eqn-all-pt-dist} holds for each $z \in U$ which satisfies
\eqbn
h_{e^{-t}}(z)  - h_1(z)  \leq 2 t - \beta \log t   , \quad \forall t \in [ \lfloor n/2 \rfloor   , n ]_{\BB Z }   .
\eqen
So, we can assume without loss of generality that there exists $t \in [\lfloor n/2 \rfloor , n]_{\BB Z}$ such that
\eqb \label{eqn-bad-time-exists}
h_{e^{-t}}(z) - h_1(z) \geq 2t - \beta   \log t .
\eqe 

Let $\tau$ be the largest time $t\in [\lfloor n/2 \rfloor , n]_{\BB Z}$ for which~\eqref{eqn-bad-time-exists} holds.  
Let $N \in  [ \lfloor n/2 \rfloor  - 1   ,  n -2 ]_{\BB Z}$ be chosen so that $\tau \in [N +1 , N+2]$.  
We will use Lemma~\ref{lem-bad-pt-dist'} to estimate $D_h\left( \bdy B_z( e^{- N - 40} ) , \bdy B_z( e^{-n/2 -1 } ) \right)$ and Lemma~\ref{lem-good-pt-dist} to estimate $D_h\left( \bdy B_z( e^{- n - 50} ) , \bdy B_z( e^{-N-2} ) \right) $. Then, we will use Lemma~\ref{lem-max-around} to ``link up" a path from $ \bdy B_z( e^{- N - 40} )$ to $ \bdy B_z( e^{-n/2 -1 })$ and a path from $\bdy B_z( e^{- n - 50} )$ to $\bdy B_z( e^{-N-2 } )$ in order to get a path from $\bdy B_z(e^{-n-50} )$ to $\bdy B_z(e^{-n/2-1} )$. 

By the definitions of $\tau$ and $N$, the condition~\eqref{eqn-bad-pt-dist-interval} holds with $N$ in place of $n$. 
Lemma~\ref{lem-bad-pt-dist'} therefore implies that
\eqb \label{eqn-use-bad-pt-dist0'}
D_h\left( \bdy B_z( e^{- N - 40} ) , \bdy B_z( e^{-N/2 -1 } ) \right) 
\leq C_1 N^{-\theta }   ,
\eqe 
where $C_1 \in (0,\infty)$ is the random variable $C$ from Lemma~\ref{lem-bad-pt-dist'}.  
Since $\lfloor n/2 \rfloor - 1 \leq N\leq n$, the relation~\eqref{eqn-use-bad-pt-dist0'} implies that also
\eqb \label{eqn-use-bad-pt-dist'}
D_h\left( \bdy B_z( e^{- N - 40} ) , \bdy B_z( e^{-n/2 -1 } ) \right) 
\leq C_1 N^{-\theta }   .
\eqe  

Since $\tau$ is the largest time in $[\lfloor n/2 \rfloor , n]_{\BB Z}$ for which~\eqref{eqn-bad-time-exists} holds, we have $h_{e^{-t}}(z)  - h_1(z) \leq 2t-\beta\log t$ for each $t\in [N+2,n]$. That is, the condition~\eqref{eqn-good-pt-dist-interval} from Lemma~\ref{lem-good-pt-dist} holds with $m = N+2$ and our given choice of $n$. Therefore, Lemma~\ref{lem-good-pt-dist} implies that
\eqb \label{eqn-use-good-pt-dist}
D_h\left( \bdy B_z( e^{- n - 50} ) , \bdy B_z( e^{-N-2 } ) \right) 
\leq C_2 N^{-\theta} ,  
\eqe 
where $C_2 \in (0,\infty)$ is a deterministic constant times the random variable $C$ from Lemma~\ref{lem-good-pt-dist}. 

By Lemma~\ref{lem-max-around}, a.s.\ there exists a random $C_3 \in (0,\infty)$ (which does not depend on $z$ or $n$) such that
\eqb \label{eqn-use-max-around}
D_h\left(\text{around $\BB A_z( e^{-N-21} , e^{-N-20})$} \right) \leq C_3  N^{- \theta } .
\eqe
The union of any path from $\bdy B_z( e^{- N - 40} ) $ to $\bdy B_z( e^{-n/2 -1 } )$, any path from $ \bdy B_z( e^{- n - 50} ) $ to $\bdy B_z( e^{-N-2 } )$, and any path in $\BB A_z( e^{-N-21} ,e^{-N-20} )$ which disconnects the inner and outer boundaries of this annulus is connected and contains a path from $\bdy B_z(e^{-n-50} )$ to $\bdy B_z(e^{-n/2-1} )$. 
Hence combining~\eqref{eqn-use-bad-pt-dist'}, \eqref{eqn-bad-pt-dist'}, and~\eqref{eqn-use-max-around} and recalling that $N \in [\lfloor n/2 \rfloor ,n]_{\BB Z}$ gives~\eqref{eqn-all-pt-dist}. 
\end{proof}

By summing the estimate of Lemma~\ref{lem-all-pt-dist} over dyadic values of $n$, we obtain the following. 

\begin{lem} \label{lem-all-scale-dist}
Let $U\subset\BB C$ be a bounded open set and let $\theta  \in (0,  \xi_\crit / 4)$.
Almost surely, there exists a random $C  \in (0,\infty)$ such that for each $z\in U$ and each $k ,\ell\in\BB N$ with $\ell  < k$, 
\eqb \label{eqn-all-scale-dist}
D_h\left(\bdy B_z( e^{-2^k} )   , \bdy B_z( e^{-2^\ell}) \right) \leq C 2^{-\theta \ell}  .
\eqe
\end{lem}
\begin{proof} 
By Lemma~\ref{lem-all-pt-dist}, a.s.\ there exists a random $C_0 \in (0,\infty)$ such that  
\eqb \label{eqn-use-all-pt-dist}
D_h\left( \bdy B_z(e^{- 2^j -50 } ) , \bdy B_z(e^{-2^{j-1}-1} ) \right) \leq C_0 2^{-\theta  j} ,\quad\forall z \in U ,\quad \forall j \in \BB N  .
\eqe
By Lemma~\ref{lem-max-around}, a.s.\ there exists a random $C_1 = C_1(\delta,\alpha) \in (0,\infty)$ such that 
\eqb \label{eqn-use-max-around'}
D_h\left( \text{around $\BB A_z( e^{-2^j-2} ,e^{-2^j-1} )$} \right) \leq C_1 2^{-\theta  j} ,\quad\forall z \in U ,\quad \forall j \in \BB N  . 
\eqe

For each $j\in\BB N$, the union of any path from $\bdy B_z( e^{- 2^{j+1} -50 } )$ to $\bdy B_z( e^{-2^{j}-1}) $, any path from $\bdy B_z(e^{- 2^j -50 } )$ to $\bdy B_z( e^{-2^{j-1}-1}  ) $, and any path in $\BB A_z( e^{-2^j-2} , e^{-2^j-1})$ which disconnects the inner and outer boundaries of this annulus is connected. Hence, if we consider paths which attain the minimal $D_h$-distances in~\eqref{eqn-use-all-pt-dist} and~\eqref{eqn-use-max-around'} for $j=\ell,\dots,k$ then the union of these paths contains a path from $ \bdy B_z(e^{-2^k} )$ to $\bdy B_z( e^{-2^\ell} )$ whose $D_h$-length is at most
\eqbn
(C_0 + C_1) \sum_{j=\ell}^k  2^{-\theta  j}  .
\eqen
This last quantity is at most $C 2^{-\theta \ell}$ for an appropriate choice of $C>0$. 
\end{proof}

\begin{proof}[Proof of Theorem~\ref{thm-holder}]
It suffices to prove~\eqref{eqn-holder} for pairs of points $z,w\in U$ such that $|z-w| \leq e^{-100}$: the general case follows by applying the triangle inequality to points $z = z_0 , z_1,\dots,z_N = w$ such that $|z_j - z_{j-1}| \leq e^{-100}$ for each $j=1,\dots,N$ and possibly increasing $C$. 

For $z,w\in U$ with $0 < |z-w| \leq e^{-100}$, let $\ell = \ell(z,w) \in \BB N$ be chosen so that $|z-w| \in [e^{-2^{\ell+2}} ,e^{-2^{\ell+1} }]$, equivalently
\eqb \label{eqn-log-dist}
\log \frac{1}{|z-w|} \in \left[ 2^{\ell+1} , 2^{\ell+2} \right] .
\eqe
By Lemma~\ref{lem-all-scale-dist}, a.s.\ there exists a random $C_0 \in (0,\infty)$ (which does not depend on $z$ or $w$) such that
\eqbn
D_h\left(\bdy B_z( e^{-2^k} ) , \bdy B_z( e^{-2^\ell} ) \right) \leq C_0 2^{-\theta \ell}  , \quad \forall k \in \BB N ,
\eqen
and the same is true with $w$ in place of $z$. 
Sending $k\rta\infty$ and using the lower semicontinuity of $D_h$ shows that a.s.\ 
\eqb \label{eqn-use-all-scale-dist}
D_h\left(z , \bdy B_z( e^{-2^\ell} ) \right) \leq C_0 2^{-\theta \ell}  ,
\eqe
and the same is true with $w$ in place of $z$. 

By Lemma~\ref{lem-max-around} a.s.\ there exists $C_1 \in (0,\infty)$ (which does not depend on $z$ or $w$) such that
\eqb \label{eqn-holder-around}
D_h\left( \text{around $\BB A_z( e^{-2^\ell-2} ,e^{-2^\ell-1} )$} \right) \leq C_1 2^{-\theta  \ell}. 
\eqe
We have $|z-w| \leq e^{-2^{\ell+1}} \leq e^{-2^\ell-2}$, so $w \in B_z( e^{-2^\ell-2} )$ and $  B_z( e^{-2^\ell-1} ) \subset B_w( e^{-2^\ell} )$. Consequently, the union of any path from $w$ to $\bdy B_w( e^{-2^\ell} )$, any path from $z$ to $\bdy B_z(e^{-2^\ell} )$, and any path in $\BB A_z( e^{-2^\ell-2} ,e^{-2^\ell-1} )$ which disconnects the inner and outer boundaries of this annulus is connected. It therefore follows from~\eqref{eqn-use-all-scale-dist} and~\eqref{eqn-holder-around}, followed by~\eqref{eqn-log-dist}, that 
\eqb
D_h(z,w) \leq (2C_0 + C_1) 2^{-\theta\ell} \leq C \left(\log\frac{1}{|z-w|} \right)^{-\theta}
\eqe
for an appropriate choice of (random) $C \in (0,\infty)$.
\end{proof}

\subsection{Proof of Proposition~\ref{prop-optimality}}
\label{sec-optimality}

Let us now prove our lower bound for the modulus of continuity of $D_h$.

\begin{proof}[Proof of Proposition~\ref{prop-optimality}]
Recall the definition of the annulus $A_z^\parallel(e^{-t}) = \BB A_z(e^{-t-100} , e^{-t})$ from~\eqref{eqn-annuli-def}. 
By Lemma~\ref{lem-all-scales} (applied with $\zeta =1/4$, say) and Definition~\ref{def-triple-max}, a.s.\ there exists a random $C \in (0,\infty)$ such that for each $n\in\BB N$ and each $z\in (e^{-n-100}\BB Z^2)\cap U$,
\allb \label{eqn-dist-lower}
D_h\left( z,\bdy B_z(e^{-n} ) \right)
&\geq D_h\left(\text{across $A_z^\parallel(e^{-n})$}\right) \notag\\
&\geq C^{-1} \left[ M_z(e^{-n}) \right]^{-1} \exp\left( \xi_\crit h_{e^{-n}}(z) - 2\xi_\crit n \right) \notag\\
&\geq C^{-1}  \exp\left( \xi_\crit h_{e^{-n}}(z) - 2\xi_\crit n - |2n - h_{e^{-n}}(z)|^{3/4}  \right) .
\alle

By Proposition~\ref{prop-gff-max}, a.s.\ for each large enough $n\in\BB N$, 
\eqb \label{eqn-optimality-max}
h_{e^{-n}}(z) \leq 2n  ,\quad\forall z\in (e^{-n-100}\BB Z^2) \cap U .
\eqe
By Proposition~\ref{prop-gff-tight}, for each $\alpha' > 3/4$ a.s.\ there are infinitely many $n\in\BB N$ for which there exists $z_n \in (e^{-n}\BB Z^2)\cap U$ which satisfies
\eqb \label{eqn-optimality-pt}
h_{e^{-n}}(z_n)  \geq 2n - \alpha' \log n .
\eqe
By applying~\eqref{eqn-optimality-max} and~\eqref{eqn-optimality-pt} to estimate the right side of~\eqref{eqn-dist-lower} for $z=z_n$, we get
\allb \label{eqn-optimality1}
D_h\left( z_n ,\bdy B_{z_n}( e^{-n} ) \right)
 \geq  C^{-1}  \exp\left( -  \xi_\crit \alpha' \log n  - |\alpha' \log n|^{3/4}  \right) .
\alle
If we are given $\theta' > 3\xi_\crit/4$ and we choose $\alpha' \in (3/4, \theta'/\xi_\crit)$, then for each large enough $n\in\BB N$, the right side of~\eqref{eqn-optimality1} is bounded below by $n^{-\theta'}$. This gives~\eqref{eqn-optimality} with $z=z_n$ and $w\in \bdy B_{z_n}(e^{-n})$.  
\end{proof}

\bibliography{cibib}
\bibliographystyle{hmralphaabbrv}

\end{document}